\documentclass[letterpaper, 12pt]{article}
\usepackage{amsmath,amsthm,amssymb,fullpage,tikz,subfig}
\usetikzlibrary{calc}

\DeclareMathOperator{\kneser}{\mathrm{KN}}
\DeclareMathOperator{\tkop}{\mathrm{TK}}
\DeclareMathOperator{\tkfop}{\mathcal{TK}}
\DeclareMathOperator{\ex}{\text{ex}}
\DeclareMathOperator{\vc}{\text{dim}}
\newcommand{\tk}[2]{\tkop^{#2}(#1)}
\newcommand{\tkf}[2]{\tkfop^{#2}(#1)}
\newcommand{\C}[2]{$C_{#1}^{#2}$}
\newcommand{\F}{\mathcal{F}}

\title{On the Chromatic Thresholds of Hypergraphs}
\author{
 J\'{o}zsef Balogh \footnote{
    University of Illinois, Urbana-Champaign, 
    jobal@math.uiuc.edu. 
    This material is based on work partly done at University of California, San Diego,  
    and at SZTE, Bolyai Institute, Szeged, Hungary. 
    Research supported by NSF CAREER Grant DMS-0745185,
    UIUC Campus Research Board Grant 11067, OTKA Grant K76099, and the Arnold O. Beckman Research Award (UIUC Campus Research Board 13039) grant, 
     Also supported by the European Union and co-funded by the European Social
 Fund  under the project ``Telemedicine-focused research activities on
 the field of Mathematics,  Informatics and Medical sciences'' of project
 number     ``T\'{A}MOP-4.2.2.A-11/1/KONV-2012-0073''.}
 \and Jane Butterfield \footnote{University of Minnesota, Minneapolis. butter@umn.edu.  Research partly supported  by the Dr. Lois M. Lackner Mathematics Fellowship and NSF grant DMS 08-38434, ``EMSW21-MCTP: Research Experience for Graduate Students''.}
 \and Ping Hu \footnote{University of Illinois, Urbana-Champaign, pinghu1@math.uiuc.edu.}
 \and John Lenz \footnote{University of Illinois at Chicago. lenz@math.uic.edu. Research partly supported by NSA Grant H98230-13-1-0224.}
 \and Dhruv Mubayi \footnote{Department of Mathematics, Statistics, and Computer Science, University of Illinois, Chicago IL 60607, email:
 mubayi@math.uic.edu. Research supported in part by  NSF Grant 0969092.}
}
\date{\today}

\begin{document}

\maketitle

\begin{abstract}

Let $\mathcal{F}$ be a family of $r$-uniform hypergraphs. The \emph{chromatic threshold}
of $\mathcal{F}$ is the infimum of all non-negative reals $c$ such that the subfamily of
$\mathcal{F}$ comprising  hypergraphs $H$ with minimum degree
at least $c \binom{\left| V(H) \right|}{r-1}$ has bounded chromatic number.  This parameter has a long history for graphs ($r=2$), and in this paper we begin its systematic study for hypergraphs.

\L uczak and Thomass\'{e} recently proved that the chromatic threshold of the so-called near bipartite graphs is zero, and our main contribution is to generalize this result to $r$-uniform hypergraphs. For this class of hypergraphs, we also show that the exact Tur\'an number is achieved uniquely by the complete $(r+1)$-partite hypergraph with nearly equal part sizes.  This is one of very few infinite families of nondegenerate hypergraphs whose Tur\'an number is determined exactly.  In an attempt to generalize Thomassen's result that the chromatic threshold of triangle-free graphs is $1/3$, we prove bounds for the chromatic threshold of the family of 3-uniform hypergraphs not containing $\{abc, abd, cde\}$, the so-called generalized triangle.

 In order to prove upper bounds we introduce the concept of \emph{fiber bundles}, which can be thought
of as a hypergraph analogue of directed graphs.  This leads to the notion of \emph{fiber bundle dimension}, a structural property of fiber bundles that is
based on the idea of Vapnik-Chervonenkis dimension in hypergraphs.  Our lower bounds follow from explicit
constructions, many  of which use a hypergraph analogue of the Kneser graph. Using methods from extremal set theory, we prove that these Kneser hypergraphs have unbounded chromatic number.  This generalizes a result of Szemer\'edi for graphs and might be of independent
interest.  Many open problems remain.

\end{abstract}

\def\drawkqq{
\begin{tikzpicture}[scale=0.5]
  \foreach \x in {0,1,2}{
    \draw[fill] (\x,0) circle (2pt);
    \draw[fill] (\x,2) circle (2pt);
    \foreach \y in {0,1,2}{
      \draw[solid] (\x,0) -- (\y,2);
    }
  }
\end{tikzpicture}
}

\tikzstyle{vertex}=[circle,fill=black,inner sep=2pt]
\tikzstyle{kqq}=[circle,draw=black,densely dashed]

\def\largedimfigure{
\begin{tikzpicture}
   \node (k1) at (0,0) [kqq] {\drawkqq};
   \node (k2) at (2.5,0) [kqq] {\drawkqq};
   \node (k3) at (7,0) [kqq] {\drawkqq};
   \draw (4.8,0) node {\dots};

   \begin{scope}[xshift=-1cm,yshift=-1.4in]

   \begin{scope}
     \node (v1) at (-0.5,0) [vertex] {};
     \node (v2) at (0,0) [vertex] {};
     \node (v3) at (0.5,0) [vertex] {};
     \draw (0,0) ellipse (0.9 and 0.3);
     \draw (0,-0.6) node {$E_1$};
   \end{scope}

   \begin{scope}[xshift=3cm]
     \node (v4) at (-0.5,0) [vertex] {};
     \node (v5) at (0,0) [vertex] {};
     \node (v6) at (0.5,0) [vertex] {};
     \draw (0,0) ellipse (0.9 and 0.3);
     \draw (0,-0.6) node {$E_2$};
   \end{scope}

   \begin{scope}[xshift=6cm]
     \node (v7) at (-0.5,0) [vertex] {};
     \node (v8) at (0,0) [vertex] {};
     \node (v9) at (0.5,0) [vertex] {};
     \draw (0,0) ellipse (0.9 and 0.3);
     \draw (0,-0.6) node {$E_3$};
   \end{scope}

   \draw (7.5,0) node {\dots};

   \begin{scope}[xshift=9cm]
     \node (v10) at (-0.5,0) [vertex] {};
     \node (v11) at (0,0) [vertex] {};
     \node (v12) at (0.5,0) [vertex] {};
     \draw (0,0) ellipse (0.9 and 0.3);
     \draw (0,-0.6) node {$E_d$};
   \end{scope}

   \end{scope}

   \draw[densely dashed] (k1) to [out=240,in=100] (v1);
   \draw[densely dashed] (k1) to [out=270,in=110] (v4);
   \draw[densely dashed] (k1) to [out=290,in=120] (v7);
   \draw[densely dashed] (k1) to [out=320,in=120] (v10);

   \draw[densely dashed] (k2) to [out=250,in=100] (v2);
   \draw[densely dashed] (k2) to [out=280,in=110] (v4);
   \draw[densely dashed] (k2) to [out=290,in=120] (v8);
   \draw[densely dashed] (k2) to [out=330,in=120] (v11);

   \draw[densely dashed] (k3) to [out=220,in=40] (v3);
   \draw[densely dashed] (k3) to [out=250,in=40] (v6);
   \draw[densely dashed] (k3) to [out=270,in=80] (v8);
   \draw[densely dashed] (k3) to [out=290,in=90] (v12);

\end{tikzpicture}
}

\tikzstyle{vertex}=[circle,fill=black,inner sep=2pt]
\tikzstyle{uvertex}=[circle,fill=gray!30,inner sep=2pt]

\def\cyclefivefigure{
\begin{tikzpicture}
   [
      rotate=90
   ]

   \filldraw[very thick, fill=gray!20 ] ellipse (1.4 and 0.3);  
   \node at (-1,0) [vertex] {};
   \node at (0,0) [vertex] {};
   \node at (1,0) [vertex] {};

   \node at (-1,1) [vertex] {};
   \node at (-1,2) [vertex] {};
   \draw (-1,1) ellipse (0.3 and 1.4);

   \node at (1,1) [vertex] {};
   \node at (1,2) [vertex] {};
   \draw (1,1) ellipse (0.3 and 1.4);

   \node at (0,2) [vertex] {};

   \draw[rounded corners=5pt] (-1,0.8) -- (-0.8,0.8) -- (0.2,1.8) -- (0.2,2.2) -- (-1.2,2.2) -- (-1.2,0.8) -- (-1,0.8);
   \draw[rounded corners=5pt] (1,0.8) -- (0.8,0.8) -- (-0.2,1.8) -- (-0.2,2.2) -- (1.2,2.2) -- (1.2,0.8) -- (1,0.8);
\end{tikzpicture}
}

\def\cycleeightfigure{
\begin{tikzpicture}
   [
     rotate=90
   ]

   \filldraw[very thick, fill=gray!20 , rounded corners=5pt] (-1,-0.8) -- (-0.8,-0.8) -- (0.2,-1.8) -- (0.2,-2.2) -- (-1.2,-2.2) -- (-1.2,-0.8) -- (-1,-0.8); 
   
   \node at (0,-2) [vertex] {};

   \node at (-1,-2) [vertex] {};
   \node at (-1,-1) [vertex] {};
   \node at (-1,0) [vertex] {};
   \draw (-1,-1) ellipse (0.3 and 1.4);

   \node at (1,-2) [vertex] {};
   \node at (1,-1) [vertex] {};
   \node at (1,0) [vertex] {};
   \draw (1,-1) ellipse (0.3 and 1.4);

   \node at (-1,1) [vertex] {};
   \node at (-1,2) [vertex] {};
   \draw (-1,1) ellipse (0.3 and 1.4);

   \node at (1,1) [vertex] {};
   \node at (1,2) [vertex] {};
   \draw (1,1) ellipse (0.3 and 1.4);

   \node at (0,2) [vertex] {};

   \draw[rounded corners=5pt] (-1,0.8) -- (-0.8,0.8) -- (0.2,1.8) -- (0.2,2.2) -- (-1.2,2.2) -- (-1.2,0.8) -- (-1,0.8);
   \draw[rounded corners=5pt] (1,0.8) -- (0.8,0.8) -- (-0.2,1.8) -- (-0.2,2.2) -- (1.2,2.2) -- (1.2,0.8) -- (1,0.8);

   \draw[rounded corners=5pt] (1,-0.8) -- (0.8,-0.8) -- (-0.2,-1.8) -- (-0.2,-2.2) -- (1.2,-2.2) -- (1.2,-0.8) -- (1,-0.8);
\end{tikzpicture}
}

\def\cycleninefigure{
\begin{tikzpicture}
   [
      rotate=90
   ]
   \filldraw[very thick, fill=gray!20 ] (0,-3) ellipse (1.4 and 0.3);  
   
   \node at (-1,-3) [vertex] {};
   \node at (0,-3) [vertex] {};
   \node at (1,-3) [vertex] {};

   \node at (-1,-2) [vertex] {};
   \node at (-1,-1) [vertex] {};
   \node at (-1,0) [vertex] {};
   \draw (-1,-2) ellipse (0.3 and 1.4);
   \draw (-1,-1) ellipse (0.3 and 1.4);

   \node at (1,-2) [vertex] {};
   \node at (1,-1) [vertex] {};
   \node at (1,0) [vertex] {};
   \draw (1,-2) ellipse (0.3 and 1.4);
   \draw (1,-1) ellipse (0.3 and 1.4);

   \node at (-1,1) [vertex] {};
   \node at (-1,2) [vertex] {};
   \draw (-1,1) ellipse (0.3 and 1.4);

   \node at (1,1) [vertex] {};
   \node at (1,2) [vertex] {};
   \draw (1,1) ellipse (0.3 and 1.4);

   \node at (0,2) [vertex] {};

   \draw[rounded corners=5pt] (-1,0.8) -- (-0.8,0.8) -- (0.2,1.8) -- (0.2,2.2) -- (-1.2,2.2) -- (-1.2,0.8) -- (-1,0.8);
   \draw[rounded corners=5pt] (1,0.8) -- (0.8,0.8) -- (-0.2,1.8) -- (-0.2,2.2) -- (1.2,2.2) -- (1.2,0.8) -- (1,0.8);
\end{tikzpicture}
}

\def\cyclefoureightfigure{
\begin{tikzpicture}
   [
      rotate=90
   ]

   \draw[very thick, fill=gray!20, rounded corners=5pt] (-1,2.2) -- (-0.8,2.2) -- (0.2,0.2) -- (0.2,-0.2) -- (-1.2,-0.2) -- (-1.2,2.2) -- (-1,2.2); 
      
   \node at (0,0) [vertex] {};
   \node at (0,6) [vertex] {};

   \foreach \x in {0,1,2,3,4,5,6}{
      \node at (-1,\x) [vertex] {};
      \node at (1,\x) [vertex] {};
   }

   \draw (-1,1.5) ellipse (0.3 and 1.9);
   \draw (1,1.5) ellipse (0.3 and 1.9);

   \draw (-1,4.5) ellipse (0.3 and 1.9);
   \draw (1,4.5) ellipse (0.3 and 1.9);

   \draw[rounded corners=5pt] (-1,3.8) -- (-0.8,3.8) -- (0.2,5.8) -- (0.2,6.2) -- (-1.2,6.2) -- (-1.2,3.8) -- (-1,3.8);
   \draw[rounded corners=5pt] (1,3.8) -- (0.8,3.8) -- (-0.2,5.8) -- (-0.2,6.2) -- (1.2,6.2) -- (1.2,3.8) -- (1,3.8);

   \draw[rounded corners=5pt] (1,2.2) -- (0.8,2.2) -- (-0.2,0.2) -- (-0.2,-0.2) -- (1.2,-0.2) -- (1.2,2.2) -- (1,2.2);
\end{tikzpicture}
}

\def\cyclefourninefigure{
\begin{tikzpicture}
   [
     rotate=90
   ]

   \filldraw[very thick, fill=gray!20]  (0,0) ellipse (1.4 and 0.3);  

   \node at (-1,0) [vertex] {};
   \node at (-0.333,0) [vertex] {};
   \node at (0.333,0) [vertex] {};
   \node at (1,0) [vertex] {};
   \node at (0,7) [vertex] {};

   \foreach \x in {1,2,3,4,5,6,7}{
      \node at (-1,\x) [vertex] {};
      \node at (1,\x) [vertex] {};
   }

   \draw (-1,1.5) ellipse (0.3 and 1.9);
   \draw (1,1.5) ellipse (0.3 and 1.9);

   \draw (-1,2.5) ellipse (0.3 and 1.9);
   \draw (1,2.5) ellipse (0.3 and 1.9);

   \draw (-1,5.5) ellipse (0.3 and 1.9);
   \draw (1,5.5) ellipse (0.3 and 1.9);

   \draw[rounded corners=5pt] (-1,4.8) -- (-0.8,4.8) -- (0.2,6.8) -- (0.2,7.2) -- (-1.2,7.2) -- (-1.2,4.8) -- (-1,4.8);
   \draw[rounded corners=5pt] (1,4.8) -- (0.8,4.8) -- (-0.2,6.8) -- (-0.2,7.2) -- (1.2,7.2) -- (1.2,4.8) -- (1,4.8);
\end{tikzpicture}
}

\def\cyclefourninecolored{
\begin{tikzpicture}
   [
     rotate=90
   ]
   
   \node at (-1,0) [vertex] {};
   \node at (-0.333,0) [uvertex] {};
   \node at (0.333,0) [uvertex] {};
   \node at (1,0) [vertex] {};
   \node at (0,7) [vertex] {};

   \node at (0, -1) {$E_1$};
   \node at (-1,-0.7) {$v_1$};
   \node at (1,-0.7) {$v_r$};
   \node at (2,1.5) {$E_2$};
   \node at (2,3) {$E_3$};
   \node at (1.6,4) {$v_{2r}$};
   \node at (0,7.7) {$v_{3r}$};
   \node at (-1.6,4) {$v_{kr}$};

   \foreach \x in {1,2,3,5,6,7}{
      \node at (-1,\x) [uvertex] {};
      \node at (1,\x) [uvertex] {};
   }
   
   \node at (-1,4) [vertex]{};
   \node at (1,4) [vertex]{};
   
   \draw[very thick]  (0,0) ellipse (1.4 and 0.3);  

   \draw (-1,1.5) ellipse (0.3 and 1.9);
   \draw (1,1.5) ellipse (0.3 and 1.9);

   \draw (-1,2.5) ellipse (0.3 and 1.9);
   \draw (1,2.5) ellipse (0.3 and 1.9);

   \draw (-1,5.5) ellipse (0.3 and 1.9);
   \draw (1,5.5) ellipse (0.3 and 1.9);

   \draw[rounded corners=5pt] (-1,4.8) -- (-0.8,4.8) -- (0.2,6.8) -- (0.2,7.2) -- (-1.2,7.2) -- (-1.2,4.8) -- (-1,4.8);
   \draw[rounded corners=5pt] (1,4.8) -- (0.8,4.8) -- (-0.2,6.8) -- (-0.2,7.2) -- (1.2,7.2) -- (1.2,4.8) -- (1,4.8);
\end{tikzpicture}
}

\def\tfivefigure{
\begin{tikzpicture}
  
   \node at (0,0) [vertex] {};		
   \node at (0,1) [vertex] {};		
   \node at (0,-1) [vertex] {};		
   \draw (0,0) ellipse (.3 and 1.4); 
   
   \node at (1,0) [vertex] {}; 		
   \node at (2,0) [vertex] {};		
   \draw[rounded corners=5pt] ( -.2, .8) -- (-.2,1.2) -- (.2, 1.2) -- (2.2, .2) -- (2.2, -.2) -- (.8, -.2) -- cycle;						
   \draw (1,0) ellipse (1.4 and .3); 
   \draw[rounded corners=5pt] (-.2, -.8) -- (.8, .2) -- (2.2, .2) -- (2.2, -.2) -- (.2, -1.2) -- (-.2, -1.2) -- cycle;						
      
\end{tikzpicture}
}   

\def\ssevenfigure{
   \begin{tikzpicture}
	\begin{scope}[ultra thick, white!50!black]
    \draw ($({2/sqrt(3)},0)$) -- ($({-2/sqrt(3)},0)$) -- (0,2) -- cycle;
    \draw ($({-2/sqrt(3)},0)$) -- ($({1/sqrt(3)},1)$);
    \draw ($({2/sqrt(3)},0)$) --  ($({-1/sqrt(3)},1)$);
    \draw (0,0) -- (0,2); 
    \draw (0, 0) arc (-90:150:2/3);
    \end{scope}

	\node at (0,2) [vertex]{};
	\node at ($({2/sqrt(3)},0)$) [vertex]{};
	\node at ($({-2/sqrt(3)},0)$) [vertex]{};
	\node at ($({-1/sqrt(3)},1)$) [vertex]{};
	\node at ($({1/sqrt(3)},1)$) [vertex]{};
	\node at (0,0) [vertex]{};
	\node at (0,2/3) [vertex]{};

   \end{tikzpicture}
}

\def\tkfourfigure{
   \begin{tikzpicture}
	\begin{scope}[ultra thick, white!50!black]
    \draw ($({2/sqrt(3)},0)$) -- ($({-2/sqrt(3)},0)$) -- (0,2) -- cycle;
    \draw ($({-2/sqrt(3)},0)$) -- (0,2/3);
    \draw ($({2/sqrt(3)},0)$) --  (0,2/3);
    \draw (0,2/3) -- (0,2); 
    \end{scope}
    
	\node at (0,2) [vertex]{};
	\node at ($({2/sqrt(3)},0)$) [vertex]{};
	\node at ($({-2/sqrt(3)},0)$) [vertex]{};
	\node at (0,2/3) [vertex]{};
	\node at (0,4/3) [vertex]{};    
	\node at ($({-1/sqrt(3)},1)$) [vertex]{};
	\node at ($({1/sqrt(3)},1)$) [vertex]{};
    \node at (0,0) [vertex]{};
    \node at ($($({-2/sqrt(3)},0)$)!.5!(0,2/3)$) [vertex]{};
    \node at ($($({2/sqrt(3)},0)$)!.5!(0,2/3)$) [vertex]{};

   \end{tikzpicture}
}


\newcommand{\mubayi}{the last author }
\newcommand{\theoremsusingbundle}{Theorems~\ref{nearkchromatic} and \ref{mainF5thm}}
\newcommand{\theoremsproofusingbundle}{proofs of Theorems~\ref{nearkchromatic} and \ref{mainF5thm} are }
\newcommand{\sectionsusingbundle}{Sections~\ref{secChromNearr} and \ref{secF5}}

\section{Introduction}

An \emph{$r$-uniform hypergraph on $n$ vertices} is a collection of $r$-subsets of $V$,
where $V$ is a set of $n$ elements. If $r=2$ then we call it a graph.
The $r$-sets in a hypergraph are called \emph{edges},
and the $n$ elements of $V$ are called \emph{vertices}.  For a hypergraph $H$ let
$V(H)$ denote the set of vertices.  We denote the set of edges by either $E(H)$ or simply
$H$. The \emph{chromatic number} of a hypergraph $H$, denoted $\chi(H)$,
is the least integer $k$ for which there exists a map $f: V(H) \rightarrow [k]$ such that
if $E$ is an edge in the hypergraph then there exist $v,u \in E$ for which $f(v) \neq f(u)$.
For a vertex $v$ in a hypergraph $H$ we let $d(v)$ denote the number of edges in $H$ that
contain $v$.  We let $\delta(H) = \min\{d(v) : v \in V(H)\}$, called the \emph{minimum degree}
of $H$.

\theoremstyle{definition}
\newtheorem*{chromaticdef}{Definition}
\theoremstyle{plain}

\begin{chromaticdef}
Let $\mathcal{F}$ be a family of $r$-uniform hypergraphs.  The \emph{chromatic threshold}
of $\mathcal{F}$, is the infimum
of the values $c \geq 0$ such that the subfamily of $\mathcal{F}$ consisting of hypergraphs
$H$ with minimum degree
at least $c \binom{\left| V(H) \right|}{r-1}$ has bounded chromatic number.
\end{chromaticdef}

We say that $F$ is a subhypergraph
of $H$ if there is an injection from $V(F)$ to $V(H)$ such that every edge in $F$
gets mapped to an edge of $H$.  Notice that this is only possible if both $H$ and $F$ are
$r$-uniform for some $r$.
If $F$ is an $r$-uniform hypergraph, then the family of
\emph{$F$-free} hypergraphs is the family of
$r$-uniform hypergraphs that do not contain $F$ as a (not necessarily induced) subhypergraph.

The study of the chromatic thresholds of graphs was motivated by a question of Erd\H{o}s and
Simonovits~\cite{cr-erdos73}: ``If $G$ is non-bipartite, what bound on $\delta(G)$ forces $G$ to
contain a triangle?''  This question was answered by Andr\'{a}sfai, Erd\H{o}s, and
S\'{o}s~\cite{cr-andrasfai74}, who showed that the answer is $2/5 \left|  V(G) \right|$, achieved by
the graph obtained from $C_5$ by replacing each edge with a copy of $K_{n/5, n/5}$. 
Andr\'{a}sfai, Erd\H{o}s, and S\'{o}s's~\cite{cr-andrasfai74} idea, i.e., blowing up a small
triangle-free graph to create a new graph with the same chromatic number and large minimum degree,
can be generalized to show that for every $k$ and $\epsilon$ there exists a triangle-free graph $G$
with $\chi(G) \geq k$ and $\delta(G) \geq (1/3-\epsilon)|V(G)|$.  This led to the following
conjecture: if $\delta(G) > (1/3 + \epsilon) \left| V(G) \right|$ and $G$ is triangle-free, then
$\chi(G) \le k_{\epsilon}$, where $k_{\epsilon}$ is a constant depending only on $\epsilon$.

Note that the conjecture is equivalent to the statement that the family of triangle-free graphs has
chromatic threshold $1/3$.  The conjecture was proven by Thomassen~\cite{cr-thomassen02}.
Subsequently, there have been three more proofs of the conjecture: one by
\L{}uczak~\cite{cr-luczak06} using the Regularity Lemma, a result of Brandt and
Thomass\'{e}~\cite{cr-brandt11} proving that one can take $k_{\epsilon} = 4$, and a recent proof by
\L{}uczak and Thomass\'{e}~\cite{cr-luczak10} using the concept of Vapnik-Chervonenkis dimension
(which is defined later in this paper).

For other graphs, Goddard and Lyle \cite{cr-goddard10} proved that the chromatic threshold
of the family of $K_r$-free graphs is $(2r-5)/(2r-3)$ while Thomassen \cite{cr-thomassen07} showed that
the chromatic threshold of the family of $C_{2k+1}$-free graphs is zero for $k \geq 2$.
Recently, \L{}uczak and Thomass\'{e}~\cite{cr-luczak10}
gave another proof that the class of $C_{2k+1}$-free graphs has chromatic threshold zero for $k \geq 2$,
as well as several other results about related families, such as Petersen graph-free graphs.
The main result of Allen, B\"{o}ttcher, Griffiths, Kohayakawa and Morris~\cite{abgkm}
is to determine the chromatic threshold of the family of $H$-free graphs for all $H$.

We finish this section with some definitions.

\theoremstyle{definition}
\newtheorem*{assorteddef}{Definition}
\theoremstyle{plain}

\begin{assorteddef}
For an $r$-uniform hypegraph $H$ and a set of vertices $S \subseteq V(H)$, let
$H[S]$ denote the $r$-uniform hypergraph consisting of exactly those edges of $H$
that are completely contained in $S$.  We call this the hypergraph \emph{induced
by} $S$.  A set of vertices $S \subseteq V(H)$ is called \emph{independent} if $H[S]$ contains no
edges and \emph{strongly independent} if there is no edge of $H$ containing at least
two vertices of $S$.
A hypergraph is $s$-partite if its vertex set can be partitioned into $s$ parts, each of
which is strongly independent.

If $\mathcal{H}$ is a family of $r$-uniform hypergraphs, then the family of
\emph{$\mathcal{H}$-free} hypergraphs is the family of $r$-uniform hypergraphs that
contain no member of $\mathcal{H}$ as a (not necessarily induced) subgraph.
For an $r$-uniform hypergraph $H$ and an integer $n$, let $ex(n,H)$ be the maximum number of edges
an $r$-uniform hypergraph on $n$ vertices can have while being $H$-free and let
\[ \pi(H) = \lim_{n\rightarrow\infty}\frac{ex(n,H)}{\binom{n}{r}}.\]
We call $\pi(H)$ the \emph{Tur\'{a}n density} of $H$. 

Let $T_{r,s}(n)$ be the complete $n$-vertex, $r$-uniform, $s$-partite hypergraph with
part sizes as equal as possible.  When $s = r$, we write $T_r(n)$ for $T_{r,r}(n)$.
Let $t_r(n)$ be the number of edges in $T_r(n)$; notice that $t_r(n)\approx\frac{r!}{r^r}\binom{n}{r}$.
We say that an $r$-uniform hypergraph $H$ is \emph{stable} with respect to $T_r(n)$ if
$\pi(H) = r!/r^r$ and for any $\epsilon > 0$ there exists some positive $\delta$ depending only on
$\epsilon$ such that if $G$ is an $n$-vertex,
$H$-free, $r$-uniform hypergraph with at least $(1-\delta)t_r(n)$ edges, then there is a partition
of $V(G)$ into $U_1, U_2, \dots, U_r$ such that all but at most $\epsilon n^r$ edges of $G$ have
exactly one vertex in each part.

Let $\tk{s}{r}$ be the $r$-uniform hypergraph obtained from the complete graph $K_s$
by enlarging each edge with $r-2$ new vertices.  The \emph{core vertices}
of $\tk{s}{r}$ are the $s$ vertices of degree larger than one.
For $s > r$, let $\tkf{s}{r}$ be the family of $r$-uniform hypergraphs such that there exists a
set $S$ of $s$ vertices where each pair of vertices from $S$ are contained together in some edge.  The set $S$
is called the set of \emph{core vertices} of the hypergraph.
For $s \leq r$, let $\tkf{s}{r}$ be the family of $r$-uniform hypergraphs
such that there exists a set $S$ of $s$ vertices where for each pair of vertices $x \neq y \in S$,
there exists an edge $E$ with $E \cap S = \left\{ x,y \right\}$ (the definition is different when $s \leq r$
so that a hypergraph consisting of  a single edge is not in $\tkf{s}{r}$).
It is obvious that $\tk{s}{r} \in \tkf{s}{r}$.
\end{assorteddef}

\section{Results}

Motivated by the above results, we investigate the chromatic thresholds of the families of
$A$-free hypergraphs for some $r$-uniform hypergraphs $A$.
One of our main results concerns a generalization of cycles to hypergraphs.  
A \emph{partial matching} is a hypergraph whose edges are pairwise disjoint (note that it can contain vertices that lie in no edge).

\theoremstyle{definition}
\newtheorem*{nearkpartite}{Definition}
\newtheorem*{criticalhypergraph}{Definition}
\theoremstyle{plain}

\begin{nearkpartite}
Let $H$ be an $r$-uniform hypergraph. We say that $H$ is \emph{near $r$-partite} if $H$ is not
$r$-partite and there exists a partition $V_1 \cup \ldots \cup V_r$ of $V(H)$ such that all edges of
$H$ either cross the partition (have one vertex in each $V_i$) or are contained entirely in $V_1$,
and in addition $H[V_1]$ is a partial matching. We call such a partition a \emph{near
$r$-partition} if it witnesses a smallest $H[V_1]$.  The edges in $H[V_1]$ of a near $r$-partition
are called the \emph{special edges}.  Say that $H$ is \emph{mono near $r$-partite} if in
addition in a near $r$-partition $H[V_1]$ contains exactly one edge.

A hypergraph $H$ is \emph{connected} if for every $x,y \in V(H)$, there exists a sequence of
hyperedges $E_1,\dots,E_t$ such that $x \in E_1$, $y \in E_t$, and $E_i \cap E_{i+1} \neq \emptyset$
for $1 \leq i \leq t-1$.  Let $H$ be an $r$-uniform hypergraph and let $X,Y$ be two disjoint sets of
vertices of $H$.  

Let $C_1, \ldots, C_t$ be the components of $H|_Y$, where $H|_Y$ is the (potentially non-uniform)
hypergraph $\left\{ A \cap Y : A \in E(H) \right\}$ and the \emph{components of $H|_Y$} are the
maximal connected induced subhypergraphs of $H|_Y$.  The vertex set $X$ is \emph{partite-extendible}
to $Y$ if there exists a partition of $X$ into $r$ strong independent sets $X_1, \ldots, X_r$ so
that for every $1 \leq i \leq t$, there do not exist $x_1 \in X_j$ and $x_2 \in X_{\ell}$ for $j
\neq \ell$ and two edges $E_1, E_2 \in E(C_i)$ such that $E_1 \cup \{ x_1 \} \in E(H)$ and $E_2 \cup
\{x_2\} \in E(H)$.  Informally, each component extends to at most one part of the partition of $X$.
\end{nearkpartite}

Our main theorem claims that for an infinite family of hypergraphs $H$ the chromatic threshold of
the family of $H$-free hypergraphs is zero.  We will demonstrate that this family of hypergraphs is infinite below,  applying this Theorem~\ref{nearkchromatic} to a type of
hypergraph cycle (see Corollary~\ref{cyclecorollary}).

\newtheorem{nearkchromatic}[thmctr]{Theorem}
\begin{nearkchromatic} \label{nearkchromatic}
Let $H$ be an $r$-uniform, near $r$-partite hypergraph with near $r$-partition $V_1, \ldots,
\linebreak[1] V_r$.  If every component, which may be a single vertex, 
of $H[V_1]$ is partite-extendible to $V_2 \cup \ldots \cup
V_r$, then the chromatic threshold of the family of $H$-free hypergraphs is zero.
\end{nearkchromatic}

One interesting aspect of the chromatic threshold of graphs, first proved by \L{}uczak and
Thomass\'{e}~\cite{cr-luczak10}, is that there exists graphs $G$ for which the chromatic threshold
of the family of $G$-free graphs is zero while the Tur\'{a}n density of $G$ is non-zero.  We show
that a similar phenomenon occurs in hypergraphs; for a subfamily of the hypergraphs considered in
Theorem~\ref{nearkchromatic} we in fact determine the exact extremal hypergraph (see
Theorem~\ref{nearkextremal}). We prove that if a mono near $r$-partite hypergraph $H$ has Tur\'{a}n
density $r!/r^r$ and is stable with respect to $T_r(n)$ (an example of such a graph is given in
Theorem~\ref{cycleiscritical}), then its unique extremal hypergraph is the complete $r$-partite
hypergraph.  Similar results occur for graphs; see Simonovits~\cite{sim}, where for critical graphs
the Erd\H{o}s-Stone Theorem~\cite{rrl-erdos46} was sharpened.

\begin{criticalhypergraph}
Let $H$ be an $r$-uniform hypergraph.  We say that $H$ is \emph{critical} if
\begin{itemize}
  \setlength{\itemsep}{1pt}
  \setlength{\parskip}{0pt}
  \setlength{\parsep}{0pt}
\item $H$ is mono near $r$-partite,
\item there exists a near $r$-partition of $H$ whose special edge has at least $r-2$ vertices of degree one,
\item $H$ is stable with respect to $T_r(n)$.
\end{itemize}
\end{criticalhypergraph}

Recall that the stability of $H$ implies that $\pi(H) = r!/r^r$.

\newtheorem{nearkextremal}[thmctr]{Theorem}
\begin{nearkextremal} \label{nearkextremal}
Let $H$ be an $r$-uniform critical hypergraph.  Then there exists some $n_0$ such that
for $n > n_0$, $T_r(n)$ is the unique $H$-free hypergraph with the most edges.
\end{nearkextremal}

\noindent A particularly interesting critical family is one that generalizes cycles to hypergraphs.
\theoremstyle{definition}
\newtheorem*{CkDef}{Definition}
\theoremstyle{plain}

\begin{CkDef}\label{CkDef}
Fix $m\geq 4$ and let 
\[
	n=\left\{ \begin{array}{rl}	r\lfloor\frac{m}{2}\rfloor + r-1 & \mbox{ if $m$ is odd,} \\
						r\frac{m}{2}  & \mbox{ if $m$ is even.} \end{array} \right.
\]
Then \C{m}{r} is the $r$-uniform hypergraph with vertices $v_1,\dots, v_n$ and edges $E_1, \dots, E_m$ such that 
\vspace{-.3cm}
\begin{enumerate}
  \setlength{\itemsep}{1pt}
  \setlength{\parskip}{0pt}
  \setlength{\parsep}{0pt}
	\item each edge contains $r$ consecutively-labeled vertices, modulo $m$, and in particular $E_1 = \{v_1,\ldots, v_r\}$,
	\item edges $E_i$ and $E_j$ intersect if and only if $i$ and $j$ are consecutive modulo $m$, 
	\item if $i$ is odd and $1<i< m$ then $|E_{i-1}\cap E_{i}| = r-1$ and $|E_i\cap E_{i+1}| = 1$.  
	\item if $m$ is odd then $|E_1\cap E_m| = 1$; if $m$ is even then $|E_1\cap E_m|=r-1$.
\end{enumerate}
\end{CkDef}

\begin{figure}[h]
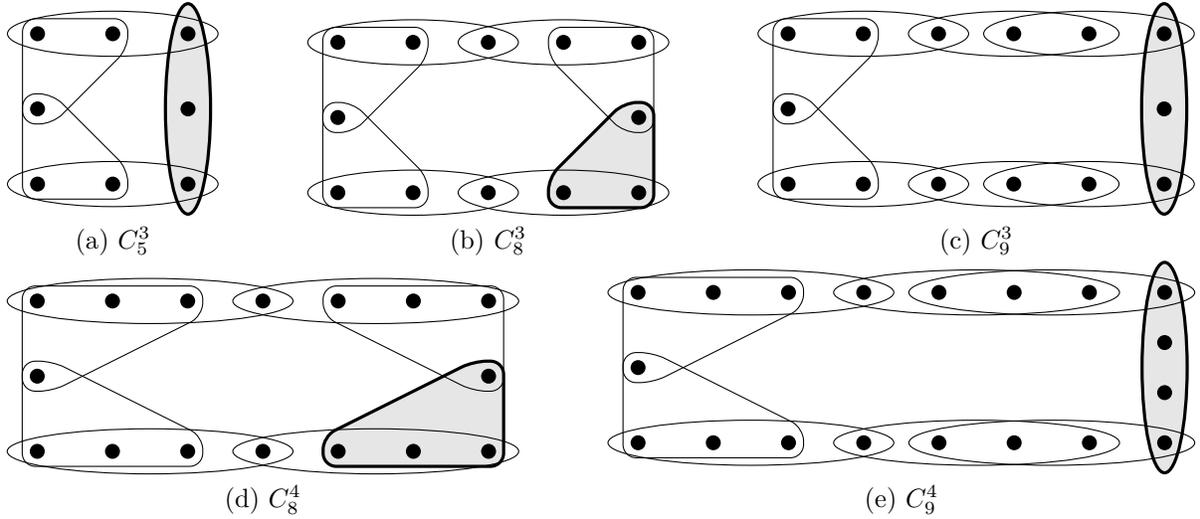

\begin{center}
\subfloat[\C{5}{3}]{\cyclefivefigure}
\hspace{0.3in}
\subfloat[\C{8}{3}]{\cycleeightfigure}
\hspace{0.3in}
\subfloat[\C{9}{3}]{\cycleninefigure}
\hspace{0.3in}
\subfloat[\C{8}{4}]{\cyclefoureightfigure}
\hspace{0.3in}
\subfloat[\C{9}{4}]{\cyclefourninefigure}
\end{center}
\caption{Hypergraph Cycles;  $E_1$ indicated in each.}
\end{figure}

\noindent We say that \C{m}{r} is \emph{odd} if $m$ is odd, and \emph{even} otherwise.  

\newtheorem{notrpartite}[thmctr]{Lemma}
\begin{notrpartite} \label{notrpartite}
  If $m = 2k+1 \geq 5$ is odd then \C{m}{r} is not $r$-partite but is mono near-$r$-partite with
  partition $V_1 = E_1 \cup \{v_{ir} : 1 \leq i \leq k\}$ and $V_j = \{v_{ir+j-1} : 1 \leq i \leq
  k+1\}$ for $2 \leq j \leq r$.  Also, every component of $C^r_m[V_1]$ is partite-extendible to
  $V_2 \cup \dots \cup V_r$.
\end{notrpartite}

\begin{proof}
Suppose $m = 2k+1$ for some integer $k$.  Notice that because $m$ is odd, we have $|E_{2k+1}\cap
E_1|=1$.  Because each edge contains consecutively-indexed vertices (modulo $m$), it follows that
$v_1$ is the common vertex.  Then $E_{2k+1}$ consists of the vertices $v_{rk+1},
v_{rk+2},\dots,v_{rk+r-1}, v_1$.  Suppose $f: V \rightarrow \{0,\dots,r-1\}$ is an $r$-coloring of
the vertices of \C{2k+1}{r} such that each color class induces a strongly independent set.  Now,
$|E_1\cap E_2|=1$ and $|E_2\cap E_3|=r-1$ (see Figure~\ref{notr_fig}).  It therefore follows that
$v_r$ is the only vertex in $E_2\setminus E_3$ and that $v_{2r}$ is the only vertex in $E_3\setminus
E_2$.  Therefore, $f(v_r)=f(v_{2r})$.  Similarly, vertices $v_r, v_{2r}, v_{3r}, \dots, v_{kr}$ all
have the same color.  Finally, $v_1=E_m\setminus E_{m-1}$ and $v_{kr}=E_{m-1}\setminus E_{m}$, and
so $f(v_1)=f(v_{kr})$.  This shows that \C{m}{r} is not $r$-partite, because $f(v_{kr})=f(v_r)$ and
$v_1,v_r$ are in $E_1$.  The hypergraph $C^r_m - E_1$ is $r$-partite via the coloring $f(v_i) = i
\pmod r$.  Also, all vertices of $E_1$ can be colored by zero to obtain a coloring where the color
classes form a near $r$-partition of \C{m}{r}.

Let $V_i$ be the vertices colored $i-1$ for $1 \leq i \leq r$.  The components of $C^r_m[V_1]$ are
the edge $E_1$ plus the single vertex components $\{v_{ir}\}$ for $2 \leq i \leq k$.  The components
of $C^r_m|_{V_2\cup\dots\cup V_r}$ (which is a $(r-1)$-uniform hypergraph) consists of a matching.
One $(r-1)$-edge of this matching is $E_2 \cap E_3$, one is $E_4 \cap E_5$, and so forth (see
Figure~\ref{notr_fig}).  First, $E_1$ is partite-extendible to $V_2 \cup \dots \cup V_r$.  Indeed,
only $E_2$ and $E_{2k+1}$ use vertices of $E_1$ and they use vertices from different components of
$C^r_m|_{V_2\cup\dots\cup V_r}$.  Also, trivially each single vertex component $\{v_{ir}\}$ is
partite-extendible to $V_2\cup\dots\cup V_r$, finishing the proof.
\end{proof}

\begin{figure}[h]
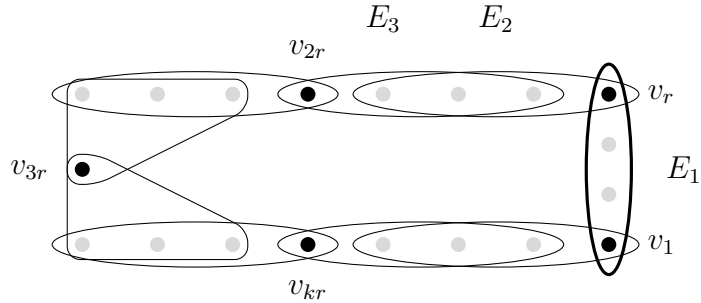

\begin{center}
	\cyclefourninecolored
\caption{Odd cycles are not $r$-partite.} \label{notr_fig}
\end{center}
\end{figure}

A theorem of Keevash and the last author~\cite{keevash04}, combined with a theorem of
Pikhurko~\cite{pikhurko08}, the supersaturation result of Erd\H{o}s and
Simonovits~\cite{rrl-erdos83}, and the hypergraph removal lemma of Gowers, Nagle, R\"{o}dl, and
Skokan \cite{rrl-gowers07,rrl-nagle06,rrl-rodl04,Rodl06,rrl-tao06} prove that \C{2k+1}{3} and
\C{2k+1}{4} are critical, see Theorem~\ref{cycleiscritical}.

For $r$ larger than four, however, \C{2k+1}{r} is not critical.  A result of Frankl and F{\"u}redi~\cite{frankl89} can easily be extended to prove that if $r \geq 5$ then 
$\pi(C^r_{2k+1})\ge \frac{1}{\binom{r}{2}e^{1+1/(r-1)}} > \frac{r!}{r^r}$.  Using techniques 
similar to those in Section~\ref{extremal}, it can in fact be shown that 
$\pi(C^5_{2k+1}) = \frac{6!}{11^4} > \frac{5!}{5^5}$ and $\pi(C^6_{2k+1}) = \frac{11\cdot 6!}{12^5} > \frac{6!}{6^6}$. 

\newtheorem{cycleiscritical}[thmctr]{Theorem}
\begin{cycleiscritical} \label{cycleiscritical}
The cycles \C{2k+1}{3} and \C{2k+1}{4} are critical for every $k \geq 2$.
\end{cycleiscritical}

Theorems~\ref{nearkchromatic}, \ref{nearkextremal}, and \ref{cycleiscritical} together with
Lemma~\ref{notrpartite} proves the following corollary, which extends the results
in~\cite{cr-thomassen07} and \cite{cr-luczak10} that the chromatic threshold of the family of
$C_{2k+1}$-free graphs is zero.

\newtheorem{cyclecorollary}[thmctr]{Corollary}
\begin{cyclecorollary} \label{cyclecorollary}
For $r = 3$ or $r = 4$ and every $k \geq 2$, there exists some $n_0$ such that for $n > n_0$,
the unique $n$-vertex, $r$-uniform, \C{2k+1}{r}-free hypergraph with the largest number of edges is $T_r(n)$.
For all $r,k \geq 2$, the chromatic threshold of the family of \C{2k+1}{r}-free hypergraphs is zero.
\end{cyclecorollary}

Note that \L{}uczak and Thomass\'{e}~\cite{cr-luczak10} proved Theorem~\ref{nearkchromatic} for graphs,
and they conjectured that the family of $H$-free graphs has chromatic threshold zero if and only if
$H$ is near acyclic and triangle free.  (A graph $G$ is \emph{near acyclic} if there exists an independent
set $S$ in $G$ such that $G-S$ is a forest and every odd cycle has at least two vertices in $S$.)
This conjecture was verified by Allen, B{\"o}ttcher, Griffiths, Kohayakawa and Morris~\cite{abgkm}.
We pose a similar question for hypergraphs.

\newtheorem{nearkchromaticconj}[thmctr]{Problem}
\begin{nearkchromaticconj}
Characterize the $r$-uniform hypergraphs $H$ for which the chromatic threshold of the family
of $H$-free hypergraphs has chromatic threshold zero.
\end{nearkchromaticconj}

Another way to generalize the triangle to $3$-uniform hypergraphs is the hypergraph
$F_5$, which is the
hypergraph with vertex set $\left\{ a,b,c,d,e \right\}$ and edges $\{a,b,c\}$, $\{a,b,d\}$,
and $\{c,d,e\}$.  Frankl and F\"{u}redi~\cite{frankle-fur} proved that $ex(n,F_5)$ is
achieved by $T_3(n)$ for $n>3000$ (recently Goldwasser~\cite{gold} has determined $ex(n, F_5)$ for all $n$).  We prove the following bounds on the chromatic threshold of
the family of $F_5$-free $3$-uniform hypergraphs.

\newtheorem{mainF5thm}[thmctr]{Theorem}
\begin{mainF5thm} \label{mainF5thm}
The chromatic threshold of the family of $F_5$-free $3$-uniform hypergraphs is
between $6/49$ and $(\sqrt{41}-5)/8 \approx 7/40$.
\end{mainF5thm}

The rest of the paper is organized as follows.
First, in Section~\ref{secFiberDef} we define and motivate fiber bundles and fiber bundle dimension,
the main tools in the proofs of Theorem~\ref{nearkchromatic} and \ref{mainF5thm}.
Next, in Section~\ref{secChromNearr} we show the power of fiber bundle dimension
by giving a relatively short proof of Theorem~\ref{nearkchromatic}.
We prove our key theorem about fiber bundle dimension, Theorem~\ref{coloringthm},
in Section~\ref{secpartitionproof}. In
Section~\ref{extremal}, we prove that \C{2k+1}{3} and \C{2k+1}{4} are 
critical (Theorem~\ref{cycleiscritical}), and then prove Theorem~\ref{nearkextremal}.
The proof of Theorem~\ref{mainF5thm} is given in Section~\ref{secF5}.  The final section gives lower bounds for
several other families of hypergraphs, along with conjectures and open problems.
The lower bounds all follow from specific constructions, some of which use a generalized
Kneser hypergraph; this graph is defined and discussed in Section~\ref{secnewKneser}. We also make a conjecture about the chromatic number of 
generalized
Kneser hypergraphs; see Conjecture~\ref{newkneserconj}.

Throughout this paper, we occasionally omit the floor and ceiling signs for the sake of clarity.

\section{Fiber Bundles and Fiber Bundle Dimension} \label{secFiberDef}

The \theoremsproofusingbundle based on a method by 
\L{}uczak and Thomass\'{e}~\cite{cr-luczak10} to color graphs, which itself was based on the 
Vapnik-Chervonenkis dimension.
Let $H$ be a hypergraph.  A subset $X$ of $V(H)$ is \emph{shattered} by $H$
if for every $Y \subseteq X$, there exists an $E \in H$ such that $E \cap X = Y$.
Introduced in~\cite{vc-sauer72} and~\cite{vc-vapnik71}, the \emph{Vapnik-Chervonenkis dimension of $H$} 
(or VC-dimension)
is the maximum size of a vertex subset shattered by $H$.

\theoremstyle{definition}
\newtheorem*{fiberdef}{Definition}
\theoremstyle{plain}

\begin{fiberdef}
A \emph{fiber bundle} is a tuple $(B, \gamma, F)$ such that $B$ is a hypergraph, $F$ is a finite set,
and $\gamma : V(B) \rightarrow 2^{2^F}$.  That is, $\gamma$ maps vertices of $B$ to collections of subsets of $F$,
which we can think of as hypergraphs on vertex set $F$.
The hypergraph $B$ is called the \emph{base hypergraph} of the
bundle and $F$ is the \emph{fiber} of the bundle.  For a vertex $b \in V(B)$, the hypergraph
$\gamma(b)$ is called the \emph{fiber over $b$}.
\end{fiberdef}

We should think about a fiber bundle as taking a base hypergraph and putting a hypergraph
``on top'' of each base vertex.
There is one canonical example of a fiber bundle.
Given a hypergraph $B$, define the
\emph{neighborhood bundle of $B$} to be the bundle $(B, \gamma, F)$ where
$F = V(B)$ and $\gamma$ maps $b \in V(B)$ to $\left\{ A \subseteq F : A \cup \left\{ b \right\} \in E(B) \right\}$.

Why define and use the language of fiber bundles?  We can consider that in some sense fiber
bundles are a generalization of directed graphs to hypergraphs, where we think of $\gamma(x)$
as the ``out-neighborhood'' of $x$.  In the neighborhood bundle, $\gamma(x)$ is related to the
neighbors of $x$ so we can consider the neighborhood bundle as some sort of directed analogue of the
undirected hypergraph $B$, where each edge is directed ``both ways''.
By thinking of the ``out-neighborhood'' of $x$ as $\gamma(x)$ and not requiring
any dependency between $\gamma(x)$ and $\gamma(y)$ for $x \neq y$, we have no dependency between
the neighborhood of $x$ and the neighborhood of $y$, which is one of the defining differences
between directed and undirected graphs.  Note that the definition of a fiber bundle differs
from the usual definition of \emph{directed hypergraph} used in the literature, which is the
reason we use the term ``fiber bundle'' instead of ``directed hypergraph.''

\newcommand{\rb}{r_B}
\newcommand{\rg}{r_{\gamma}}

A fiber bundle $(B,\gamma,F)$ is \emph{$(\rb,\rg)$-uniform} if $B$ is an $\rb$-uniform hypergraph
and $\gamma(b)$ is an $\rg$-uniform hypergraph for each $b \in V(B)$.
Given $X \subseteq V(B)$, the \emph{section of $X$} is the hypergraph with vertex set $F$
and edges $\cap_{x \in X} \gamma(x)$.  In other words, the section of $X$ is the collection 
of subsets of $F$ that 
appear in the fiber over $x$ for every $x \in X$.  Motivated by a definition of \L{}uczak and Thomass\'{e}
\cite{cr-luczak10}, we define the \emph{$H$-dimension of a fiber bundle}.
Let $H$ be a hypergraph and define $\vc_H(B, \gamma, F)$ to be the maximum integer $d$ such
that there exist $d$ disjoint edges $E_1, \ldots, E_d$ of $B$ (i.e. a matching) such that
for every $x_1 \in E_1, \ldots, x_d \in E_d$, the section of $\left\{ x_1, \ldots, x_d \right\}$
contains a copy of $H$.  Our definition of dimension coincides with the
definition of paired VC-dimension in \cite{cr-luczak10} when $(B,\gamma,F)$ is $(2,1)$-uniform and
$H = \left\{ \left\{ x \right\} \right\}$, the complete $1$-uniform, $1$-vertex hypergraph.

Let $A$ be an $r$-uniform hypergraph.  Our method of proving an upper bound on the
chromatic threshold of the family of $A$-free hypergraphs, used in \theoremsusingbundle,
is the following.  Let $G$ be an $A$-free $r$-uniform hypergraph with minimum degree
at least $c \binom{\left| V(G) \right|}{r-1}$.
We now need to show that $G$ has bounded chromatic number, which we do in two steps.
Let $(G,\gamma,F)$ be the neighborhood bundle
of $G$.  First, we show that the dimension of $(G, \gamma, F)$
is bounded by showing that if the dimension is large then we can find $A$ as a subhypergraph.
Then, given that $\dim_H(G,\gamma,F)$ is bounded, we use the following theorem to bound the chromatic number of $G$.
In most applications, we will let $H$ be an $(r-1)$-uniform, $(r-1)$-partite hypergraph.

\newtheorem{coloringthm}[thmctr]{Theorem}
\begin{coloringthm} \label{coloringthm}
Let $\rb \geq 2$, $\rg \geq 1$, $d \in \mathbb{Z}^{+}$, $0 < \epsilon < 1$, and $H$ be an
$\rg$-uniform hypergraph with zero Tur\'{a}n density.  Then there exists constants $K_1 = K_1(\rb,\rg,d,\epsilon,H)$
and $K_2 = K_2(\rb,\rg,d,\epsilon,H)$ such that the following holds.
Let $(B,\gamma,F)$ be any $(\rb,\rg)$-uniform fiber bundle where $\vc_H(B,\gamma,F) < d$ and for all $b \in V(B)$,
\begin{align*}
\left| \gamma(b) \right|  \geq \epsilon \binom{\left| F \right|}{\rg}.
\end{align*}
If $\left| F \right| \geq K_1$, then $\chi(B) \leq K_2$.
\end{coloringthm}

The above theorem is sufficent for our purposes, but our proof of Theorem~\ref{coloringthm} proves
something slightly stronger.  The conclusion of the above theorem can be reworded to say that 
either $F$ is small, the chromatic number of $B$ is bounded,
or $\dim_H(B,\gamma,F)$ is large, which means that we can find $d$ hyperedges $E_1, \ldots, E_d$ such that
every section of $x_1 \in E_1, \ldots, x_d \in E_d$ contains a copy of $H$.  In fact, the proof shows that
if $F$ is large and the chromatic number of $B$ is large,
we can guarantee not only one copy of $H$ but at least $\Omega(\left| F \right|^h)$ copies of $H$ in each section,
where $h$ is the number of vertices in $H$.

We conjecture a similar statement for all $\rg$-uniform hypergraphs $H$, instead
of just those hypergraphs with a Tur\'{a}n density of zero.

\newtheorem{coloringconj}[thmctr]{Conjecture}
\begin{coloringconj} \label{coloringconj}
Let $\rb \geq 2$, $\rg \geq 1$, $d \in \mathbb{Z}^{+}$, $0 < \epsilon < 1$, and $H$ be an $\rg$-uniform hypergraph.
Then there exists a constants $K_1 = K_1(\rb,\rg,d,\epsilon,H)$ and
$K_2 = K_2(\rb,\rg,d,\epsilon,H)$ such that the following holds.
Let $(B,\gamma,F)$ be any $(\rb,\rg)$-uniform fiber bundle where $\vc_H(B,\gamma,F) < d$ and for all $b \in V(B)$,
\begin{align*}
\left| \gamma(b) \right|  \geq \left( \pi(H) + \epsilon \right) \binom{\left| F \right|}{\rg}.
\end{align*}
If $\left| F \right| \geq K_1$, then $\chi(B) \leq K_2$.
\end{coloringconj}

The motivation behind defining and using the language of fiber bundles rather than
using the language of hypergraphs is that in the course of the proof of 
Theorem~\ref{coloringthm}, we will modify $B$ and $\gamma$ and apply induction.
As mentioned above, fiber bundles can be thought of as a directed version of a hypergraph.
When applying Theorem~\ref{coloringthm} in \sectionsusingbundle,
we start with the neighborhood bundle, which carries no ``extra'' information beyond just the hypergraph $B$.
But if we tried to prove Theorem~\ref{coloringthm} in the language of hypergraphs, we would run into trouble
when we needed to modify $\gamma$.  In the neighborhood bundle, $\gamma$ is related
to the neighborhood of a vertex and if we restricted ourselves to neighborhood bundles
or just used the language of hypergraphs, modifying $\gamma(x)$ would imply that some
$\gamma(y)$'s would change at the same time.  The notion of a fiber bundle allows us
to change the ``out-neighborhood'' of $x$ independently of changing
the ``out-neighborhood'' of $y \neq x$, and this power is critical in the proof of
Theorem~\ref{coloringthm}.

\section{Chromatic threshold for near $r$-partite hypergraphs} \label{secChromNearr}

In this section we show an application of Theorem~\ref{coloringthm} by proving Theorem~\ref{nearkchromatic}.  
Recall that $H$ is an $r$-uniform, near $r$-partite hypergraph with near $r$-partition $V_1, \dots, V_r$ 
such that every component of $H[V_1]$ is partite-extendible to $V_2\cup\dots\cup V_r$. 
Fix $\epsilon > 0$ and 
let $G$ be an $n$-vertex, $r$-uniform, $H$-free hypergraph with
$\delta(G) \geq \epsilon \binom{n}{r-1}$.  We would like to use Theorem~\ref{coloringthm}
to bound the chromatic number of $G$, so we need to choose an appropriate bundle.  We will not use
the neighborhood bundle of $G$, but a closely related bundle.
Once we have defined this bundle, we show it has bounded dimension by proving that if the
dimension is large then we can find a copy of $H$ in $G$.

\begin{proof}[Proof of Theorem~\ref{nearkchromatic}]
Let $H$ be an $r$-uniform, near $r$-partite, $h$-vertex hypergraph and let $\epsilon > 0$ be fixed.
Let $V_1, \ldots, V_r$ be a near $r$-partition of $H$ and assume every component of $H[V_1]$ is
partite-extendible to $V_2 \cup \ldots \cup V_r$.  Let
\begin{align*}
  d = |V_1|.
\end{align*}
Let $G$ be an $n$-vertex, $H$-free hypergraph with
$\delta(G) \geq \epsilon \binom{n}{r-1}$.  We need to show that
the chromatic number of $G$ is bounded by a constant depending only on $\epsilon$
and $H$.

First, choose a partition $X_1, \ldots, X_r$ of $V(G)$ such that the sizes of $X_1, \ldots, X_r$ are
as equal as possible and for every $x \in V(G)$ the number of edges containing $x$ and one vertex
from each $X_i$ is at least $\frac{1}{2r^r} \epsilon \binom{n}{r-1}$.  (Almost every
nearly-equitable partition has this property.)  We will show how to bound the chromatic number of
$G[X_1]$; the same argument can be applied to bound the chromatic number of each $G[X_i]$ and thus
the chromatic number of $G$.

Define the $(r,r-1)$-uniform fiber bundle $(B,\gamma,F)$ as follows.
Let $B = G[X_1]$, let $F = X_2 \cup \ldots \cup X_r$, and for $x \in X_1$ define
\begin{align*}
\gamma(x) = \left\{ \left\{ x_2, \ldots, x_r \right\} \subseteq F : x_2 \in X_2, \ldots, x_r \in X_r, 
          \left\{ x, x_2, \ldots, x_r \right\} \in G \right\}.
\end{align*}
Then $\gamma(x)$ has size at least $\frac{1}{2r^r} \epsilon \binom{n}{r-1}$.  Let $L$ be the complete
$(r-1)$-uniform, $(r-1)$-partite hypergraph on $(rh)^h$ vertices with color classes of (nearly) equal sizes.  
Using that the Tur\'{a}n density of $L$ 
is zero, 
we apply Theorem~\ref{coloringthm} to show that there exists constants 
$K_1 = K_1(r,\epsilon,H)$ and $K_2 = K_2(r,\epsilon,H)$ such that one of the
following holds:  either  $\left| F \right| \leq K_1$, $\chi(B) \leq K_2$, or $\vc_L(B,\gamma,F) \geq d$.
Since $\left| F \right| = (1-1/r) \left| V(G) \right|$, if $\left| F \right| \leq K_1$ then $|V(G)|<K_1\left(\frac{r}{r-1}\right)$; 
therefore, if either of the first
two possibilities occur then the chromatic number of $G[X_1]$ is bounded.  We may therefore assume that $\vc_L(B,\gamma,F) \geq d$.  

\begin{figure}
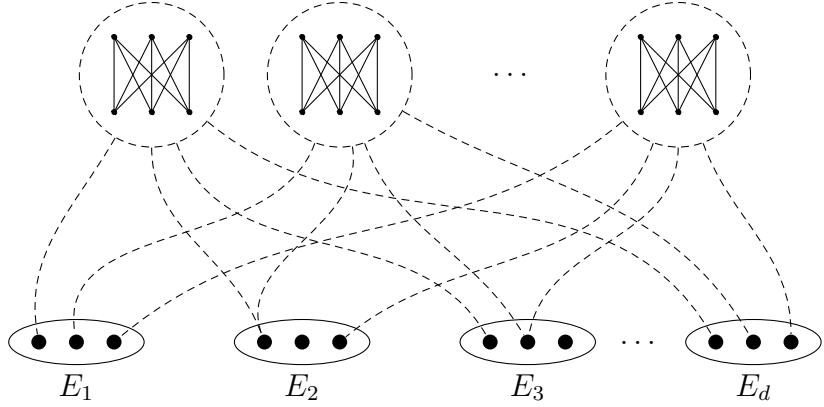

\begin{center}
\largedimfigure
\end{center}
\caption{The structure guaranteed by dimension $d$.}
\label{nearkfigure}
\end{figure}

We now show that if $\vc_L(B,\gamma,F)\geq d$ then 
$G$ contains a copy of $H$, which follows from the definition of near
$r$-partite and partite-extendible.  Since $\vc_L(B,\gamma,F) \geq d$, there are $d$ edges $E_1,
\ldots, E_d$ such that for each $x_1 \in E_1, \ldots, x_d \in E_d$, we have that $\gamma(x_1) \cap
\ldots \cap \gamma(x_d)$ contains a copy of $L$; see Figure~\ref{nearkfigure}.  Since $h = \left| V(H) \right|$, from each $\gamma(x_1)
\cap \ldots \cap \gamma(x_d)$ we can pick a copy of the complete $(r-1)$-uniform, $(r-1)$-partite
hypergraph on $h$ vertices whose color classes are of nearly equal size 
so that all these copies are vertex disjoint.  Assume $V_1 = A_1 \cup
\ldots \cup A_{\ell} \cup \left\{ a_{\ell+1} \right\} \cup \ldots \cup \left\{ a_{\ell'} \right\}$,
where $A_1, \ldots, A_{\ell}$ are the special edges of $H$.  Because $\ell\leq\ell'\leq d$, we can embed a copy of $H$ in $G$ by
mapping $A_i$ to $E_i$ for $1 \leq i \leq \ell$, mapping $a_i$ to any vertex in $E_i$ for $\ell+1
\leq i \leq \ell'$, and mapping the components of $H|_{V_2 \cup \ldots \cup V_r}$ to the complete
$(r-1)$-uniform, $(r-1)$-partite hypergraphs as follows.

Consider some component $C$ in $H|_{V_2 \cup \ldots \cup V_r}$.  Any such $C$ is an $(r-1)$-uniform,
$(r-1)$-partite hypergraph on at most $h$ vertices.  Let $D_1, \ldots, D_{\ell'}$ be the components of
$H[V_1]$; $D_i$ is either one of the special edges $A_1, \ldots, A_{\ell}$ or $D_i$ consists only of
the vertex ${a_i}$ for some $\ell+ 1 \leq i \leq \ell'$.  Since $V(D_i)$ is partite-extendible to
$V_2 \cup \ldots \cup V_r$, edges in $C$ extend to at most one vertex $z_i \in D_i$.  Since vertices
in $V_1$ are embedded to vertices in $E_1, \ldots, E_d$, this means that $C$ must be embedded in
$\gamma(x_1) \cap \ldots \cap \gamma(x_d)$ for some $x_i \in E_i$.  It is crucial that $C$ does not
need to be embedded in $\gamma(x) \cap \gamma(y)$ for $x \neq y \in E_i$; this is what is guaranteed
by the definition of partite-extendible.  Embedding $C$ is possible since $\gamma(x_1) \cap \ldots
\cap \gamma(x_d)$ contains a complete $(r-1)$-uniform, $(r-1)$-partite hypergraph on $h$ vertices
and $h = |V(H)|$ (so even if more than one component is embedded in the same $\gamma(x_1) \cap
\ldots \cap \gamma(x_d)$, there is enough room for both of them.)
\end{proof}

\section{Coloring hypergraphs with bounded dimension} \label{secpartitionproof}

\theoremstyle{definition}
\newtheorem*{condition1}{Condition 1}
\theoremstyle{plain}

\newcommand{\cBoost}{\eta}
\newcommand{\cGreedy}[1]{\psi_{#1}}
\newcommand{\cRefineS}{\alpha}
\newcommand{\cSforceF}{\beta}
\newcommand{\cMinF}{K_1}
\newcommand{\cChromaticBound}{K_2}

In this section, we will prove Theorem~\ref{coloringthm}.  To prove Theorem~\ref{coloringthm}, given
a fiber bundle $(B,\gamma,F)$ satisfying the conditions of the theorem, we must show how to produce
a proper coloring of $B$ with a bounded number of colors.  We do this via a partition refinement
strategy.  Below, we give an algorithm to refine a partition of $(B,\gamma,F)$ (a partition is
formally defined below).  The algorithm will increase a density measure (also defined below) by a
constant amount and add a constant number of new parts, so the refinement will halt after a constant
number of iterations.  Each part of the resulting partition will either correspond to an independent
set in $B$ or to a vertex set $X$ where $B[X]$ has a maximal matching of bounded size (so $B[X]$ has
bounded chromatic number), therefore producing a proper coloring of $B$ with a bounded number of
colors.

Throughout this section, fix $\rb \geq 2$, $\rg \geq 1$, $d \in \mathbb{Z}^{+}$, $0 < \epsilon <
\frac{1}{4} \rb^{-d}$, and $H$ an $\rg$-uniform hypergraph with zero Tur\'{a}n density.

\begin{condition1}
Let $(B,\gamma,F)$ be an $(\rb,\rg)$-uniform fiber bundle for which $\vc_H(B,\gamma,F) < d$ and if
$b \in V(B)$, then $\left| \gamma(b) \right| \geq \epsilon \binom{\left| F \right|}{\rg}$.
\end{condition1}

Define the following constants.

\begin{align*}
  \cRefineS = \frac{1}{1000} \left(\frac{\epsilon}{4^{\rb^d + 1}}\right)^{d+1},   \qquad \cBoost = \frac{1}{4} \epsilon^2 \cRefineS,
  \qquad \cSforceF = \cRefineS^{1/\cBoost}, \qquad \cChromaticBound = \left\lceil \rb d(\rb^d+2)^{1/\cBoost} \right\rceil.
\end{align*}
Next, pick $\cMinF$ large enough so that if $\left| F \right| \geq \cMinF$ and $S \subseteq \binom{F}{\rg}$ with
$\left| S \right| \geq \alpha \beta \epsilon \binom{\left| F \right|}{\rg}$, then $S$ contains a copy of $H$.

If $(B,\gamma,F)$ is a fiber bundle, a \emph{partition} $P$ of $(B,\gamma,F)$ is a family
$P = \{ (X_1,S_1), \linebreak[1] \ldots, (X_p,S_p) \}$ such that $X_1, \ldots, X_p$ is a partition of $V(B)$
and $S_1, \ldots, S_p$ is a partition of $\binom{F}{\rg}$, where we allow $X_i = \emptyset$ or $S_i = \emptyset$.
A partition $Q$ is a refinement of a partition $P$ if
for each $(X,S) \in P$, there exist $(Y_1,T_1), \ldots, (Y_q,T_q) \in Q$ such that $X = \cup Y_i$ and $S = \cup T_i$.
For $X \subseteq V(B)$ and $S \subseteq 2^F$, the \emph{density of $(X,S)$} is
\begin{align*}
d(X,S) = \begin{cases}
1 & S = \emptyset \text{ or } X = \emptyset, \\
\min \left\{ \frac{\left| \gamma(x) \cap S \right|}{\left| S \right|} : x \in X \right\} & \text{otherwise},
\end{cases}
\end{align*}
and define
\begin{align*}
d(P) = \min\left\{  d(X,S) : (X,S) \in P \right\}.
\end{align*}
A partition $P$ is a \emph{partial coloring} if for every $(X,\emptyset) \in P$
we have that $B[X]$ is independent.
The \emph{rank} of a partition $P$ is
the minimum of $\left| S \right|$ over all $(X,S) \in P$ with $S \neq \emptyset$.

The key lemma in this section is the following.

\newtheorem{partitionlem}[thmctr]{Lemma}
\begin{partitionlem} \label{partitionlem}
Let $(B,\gamma,F)$ be a fiber bundle satisfying Condition~1 and $\left| F \right| \geq \cMinF$.
Let $X \subseteq V(B)$ and $S \subseteq \binom{F}{\rg}$ with $X \neq \emptyset$, $d(X,S) \geq \epsilon$, and
$\left| S \right| \geq \cSforceF \binom{\left|  F  \right|}{\rg}$.
Then there exists a partition $Y_1, \ldots, Y_q,Z$ of $X$
and a partition $T_1, \ldots, T_q$ of $S$ such that $q \leq \rb^d+1$ and
\begin{itemize}
\item $\left| T_i \right| \geq \cRefineS \left| S \right|$,
\item $d(Y_i,T_i) \geq \min \left\{ 1, \cBoost + d(X,S)\right\}$,
\item $B[Z]$ is independent.
\end{itemize}
\end{partitionlem}

\noindent This lemma has an easy corollary.

\newtheorem{refinementlem}[thmctr]{Corollary}
\begin{refinementlem} \label{refinementlem}
Let $(B,\gamma,F)$ be a fiber bundle satisfying Condition~1 and $\left| F \right| \geq \cMinF$.
Let $P$ be a partial coloring of $(B,\gamma,F)$ where 
$P$ has rank at least $\cRefineS^{k} \binom{\left|F\right|}{\rg}$ with $k \leq \frac{1}{\cBoost}$.
Then there exists a refinement $Q$ of $P$ such that
\begin{itemize}
\item $\left| Q \right| \leq (\rb^d+2) \left| P \right|$,
\item $Q$ is also a partial coloring,
\item the rank of $Q$ is at least $\cRefineS^{k+1} \binom{\left|F\right|}{\rg}$,
\item $d(Q) \geq \min\left\{1, \cBoost + d(P)\right\}$.
\end{itemize}
\end{refinementlem}

\begin{proof}
For each pair $(X,S) \in P$ with $X \neq \emptyset$ and $S \neq \emptyset$, apply
Lemma~\ref{partitionlem}.  Since $k \leq \frac{1}{\cBoost}$, 
$\left| S \right| \geq \cRefineS^{k} \binom{\left| F \right|}{\rg} 
                \geq \cRefineS^{1/\cBoost} \binom{\left| F \right|}{\rg} 
                \geq \beta \binom{\left| F \right|}{\rg}$.
Lemma~\ref{partitionlem} produces
$Y_1, \ldots, Y_q,Z$ and $T_1, \ldots, T_q$ with $q \leq \rb^d+1$.  We replace the pair $(X,S)$ with the pairs
$(Y_1,T_1), \ldots, (Y_q,T_q), (Z,\emptyset)$.  The resulting partition satisfies all the required properties.  
\end{proof}

We can now easily prove Theorem~\ref{coloringthm}.

\begin{proof}[Proof of Theorem~\ref{coloringthm}]
By assumption, $(B,\gamma,F)$ satisfies Condition~1.
Start with the partition $P = \left\{ (V(B), \binom{F}{\rg}) \right\}$ and
apply Corollary~\ref{refinementlem} repeatedly until the partition satisfies $d(P) = 1$.  Since the
value of $d(P)$ increases by $\cBoost$ at each step, the partition is refined at most $1/\cBoost$ times,
and so
the resulting partition $P$ has at most $(\rb^d+2)^{1/\cBoost}$ parts.
Consider a part $(X,S) \in P$.
If $S = \emptyset$, then since $P$ is a partial coloring $B[X]$ must be independent, so $\chi(B[X]) = 1$.
If $S \neq \emptyset$ then because the partition was refined at most $1/\cBoost$ times 
we know that $\left| S \right| \geq \cSforceF \binom{\left| F \right|}{\rg}$, which by the choice of $\cSforceF$ and $\cMinF$
forces a copy of $H$ in $S$.  Since $d(X,S) = 1$ we must have $S \subseteq \gamma(x)$ for every $x \in X$,
so that a matching of size $d$ in $B[X]$ witnesses that $\vc_H(B,\gamma,F) \geq d$.  Therefore, the maximum size
of a matching in $B[X]$ is $d-1$. Since the size of a maximal
matching in $B[X]$ is $d-1$, it follows that $\chi(B[X]) \leq \rb(d-1)+1$.
This implies that the chromatic number of $B$ is at most $\rb d(\rb^d+2)^{1/\cBoost}$.
\end{proof}

\newcommand{\minsecd}{\delta}

All that remains is to prove Lemma~\ref{partitionlem}.  Before proving this lemma, we make
some definitions.
If $E_1, \ldots, E_t \in B$ and $S \subseteq \binom{F}{\rg}$, then
the \emph{minimum section density of $E_1, \ldots, E_t$ with respect to $S$} is
\begin{align*}
\minsecd(E_1, \ldots, E_t,S) = \min\left\{ 
    \frac{\left| \gamma(x_1) \cap \ldots \cap \gamma(x_t) \cap S \right|}{\left| S \right|}
    : x_1 \in E_1, \ldots, x_t \in E_t \right\}.
\end{align*}
Notice that if $E_1, \ldots, E_d$ are disjoint, $\minsecd(E_1, \ldots, E_d,S) > 0$, 
$S$ contains a constant fraction of $\binom{F}{\rg}$, and $F$ is large,
then $E_1, \ldots, E_d$ witness that $\vc_H(B,\gamma,F) \geq d$.
Define constants $\cGreedy{1}, \ldots, \cGreedy{d}$ recursively by $\cGreedy{1} = 1$ and 
$\cGreedy{i+1} = \frac{1}{2} 4^{-\rb^d} \epsilon \cGreedy{i}$ for $1 \leq i \leq d-1$.

\begin{proof}[Proof of Lemma~\ref{partitionlem}]
Start by greedily selecting disjoint edges $E_1, \ldots, E_i$ of $B[X]$ such that
$\minsecd(E_1, \ldots, \linebreak[1] E_i,S) \geq \epsilon \cGreedy{i}$.  
Since for every $x \in X$
\begin{align*}
\frac{\left| \gamma(x) \cap S \right|}{\left| S \right|} \geq d(X,S) \geq \epsilon \cGreedy{1},
\end{align*}
the greedy algorithm can start with any edge $E_1$ in $B[X]$.
Assume the greedy algorithm has selected
$E_1, \ldots, E_m$ with $\minsecd(E_1, \ldots, E_m,S) \geq \epsilon \cGreedy{m}$ but for every other
edge $E$ in $B[X]$
disjoint from $E_1, \ldots, E_m$, we have $\minsecd(E_1, \ldots, E_m, E,S) < \epsilon \cGreedy{m+1}$.

First, we prove that $\vc_H(B,\gamma,F) \geq m$.
Let $m' = \min\{m,d\}$.
Since $\minsecd(E_1, \ldots, E_{m'},S) \geq \epsilon \cGreedy{m'} \geq \epsilon \cGreedy{d}$, we have
that every section of $x_1 \in E_1, \ldots, x_{m'} \in E_{m'}$ has size at least 
$\epsilon \cGreedy{d} \left| S \right| \geq \epsilon \alpha |S| \geq \alpha \epsilon \beta \binom{\left| F \right|}{\rg}$.
By the choice of $\cMinF$, the section of $x_1, \ldots, x_{m'}$ contains a copy of $H$, and so 
$m' < d$ and $m'= m$. Then $E_1, \ldots, E_m$ witness that $\vc_H(B,\gamma,F) \geq m$.

We make the following definitions.

\begin{itemize}
\item Let $R_1, \ldots, R_t$ be all $\rb^m$ sections of $v_1 \in E_1, \ldots, v_m \in E_m$ intersected with $S$.

\item Now remove elements from each $R_i$ to form $T_i$ via the following steps:
  \begin{itemize}
    \item Start with $T_i = R_i$ for all $1 \leq i \leq t$.
    \item If there exists some $i \neq j$ with $T_i \cap T_j \neq \emptyset$, divide $T_i \cap T_j$
      into two sets $A$ and $B$ with size as equal as possible and remove $A$ from $T_i$ and remove $B$ from
      $T_j$.  Repeat this until $T_1, \dots, T_t$ are pairwise disjoint.
    \item Remove elements of $T_i$ arbitrarily until
    $\left| T_i \right| < 2 \epsilon \left| S \right|$.  (If $T_i$ is already smaller than $2 \epsilon \left| S \right|$,
     nothing needs to be removed.)
  \end{itemize}
\item Let $T_{t+1} = S \setminus T_1 \setminus \ldots \setminus T_t$.
\item For $1 \leq i \leq t+1$, define
 \begin{align*}
  Y_i = \left\{ x \in X : \frac{\left| \gamma(x) \cap T_i \right|}{\left| T_i \right|} 
                          \geq \min\left\{1, \cBoost + d(X,S) \right\} \right\}.
 \end{align*}
 If some $x$ appears in more than one $Y_i$, remove it from all but the least-indexed $Y_i$.
\item Let $Z = X \setminus Y_1 \setminus \ldots \setminus Y_{t+1}$.
\end{itemize}
 By the definition of $Y_i$, $d(Y_i,T_i) \geq \min\{1, \cBoost + d(X,S)\}$.
Therefore, to finish the proof we need to check that $\left| T_i \right| \geq \cRefineS \left| S \right|$ and $B[Z]$ is independent. 

\medskip
\noindent \emph{Claim 1:} $\left| T_i \right| \geq 2\cGreedy{m+1} |S| \geq \cRefineS \left| S \right|$ for all $1 \leq i \leq t+1$.

\begin{proof}
Since $\minsecd(E_1, \ldots, E_m,S) \geq \epsilon \cGreedy{m}$, each $R_i$ has size at least
$\epsilon \cGreedy{m} \left| S \right|$ so initially each $T_i$ has size at least $\epsilon
\cGreedy{m} |S|$.  Now consider how many elements are removed from $T_i$ for some fixed $i$.  For
each $j \neq i$, half of $T_i \cap T_j$ will be removed from $T_i$ so even if $T_i$ is contained
inside $T_j$, at most half of $T_i$ will be removed.  To deal with the case when $T_i \cap T_j$ is
odd, certainly the size of $T_i$ is cut down to at most one-fourth.
There are $t-1 = \rb^m  - 1 \leq
\rb^d$ of these potential removals, so after making $T_1, \dots, T_t$ disjoint,
\begin{align*}
  \left| T_i \right| \geq \frac{1}{4^{\rb^d}} \left| R_i \right| \geq \frac{\epsilon \cGreedy{m}}{4^{\rb^d}}
  \left| S \right| = 2\cGreedy{m+1} \left| S \right|.
\end{align*}
Finally,
since $\cGreedy{1} = 1$ and $m \geq 1$, $\cGreedy{m+1} < \frac{\epsilon}{4}$, we have that
$2\cGreedy{m+1} |S| < 2\epsilon |S|$, so if after making $T_1, \dots, T_t$ disjoint, $T_i$ is still
larger than $2\epsilon |S|$, cutting $T_i$ down to size $2\epsilon |S|$ still preserves that $|T_i|
\geq 2\cGreedy{m+1} |S|$.  By the choice of constants, $2\cGreedy{m+1} \geq \cRefineS$ so $|T_i| \geq
\cRefineS |S|$.

Now consider the size of $T_{t+1}$.  Since each $T_i$ with $i \leq t$ has size at most $2 \epsilon \left| S \right|$
and we assumed that $\epsilon < \frac{1}{4} t^{-1}$ in Condition~1, the set $T_{t+1}$ has at least
$\frac{1}{2} \left| S \right| \geq 2\cGreedy{m+1}|S| \geq \cRefineS \left| S \right|$ elements.
\end{proof}

\medskip
\noindent \emph{Claim 2:} $B[Z]$ is independent.

\begin{proof}
Assume $E$ is an edge in $B[Z]$.  We would like to show that
there exists some $x \in E$ and some $T_j$ such that
\begin{align} \label{eqdensityTj}
\frac{\left| \gamma(x) \cap T_j \right|}{\left| T_j \right|} \geq \min\left\{1, \cBoost + d(X,S) \right\},
\end{align}
since this would show that $x \in Y_j$, contradicting that $x \in Z$.  Assume $E$ intersects some
$E_i$ for some $1 \leq i \leq m$, with $x \in E \cap E_i$.  Since $x \in E_i$ there is a section
$R_j$ that \emph{selects} $x$, by which we mean that $R_j$ was formed by choosing $x$ from $E_i$.  Fix some
such section $R_j$ that selects $x$, in which case $R_j \subseteq
\gamma(x)$.  Then $T_j \subseteq R_j \subseteq \gamma(x)$ and $\left| \gamma(x) \cap T_j
\right|/\left| T_j \right| = 1$ so \eqref{eqdensityTj} is satisfied.

Now assume $E$ is disjoint from $E_1, \ldots, E_m$.  Since the greedy algorithm could not continue,
$\minsecd(E_1, \ldots, E_m, E, S) < \epsilon \cGreedy{m+1}$, which implies that there exists some
$v_1 \in E_1, \ldots, v_m \in E_m, x \in E$ such that
\begin{align*}
\left| \gamma(v_1) \cap \ldots \cap \gamma(v_m) \cap \gamma(x) \cap S \right| <
\epsilon \cGreedy{m+1} \left| S \right|.
\end{align*}
By the definition of $T_i$, there exists some $T_i$ such that $T_i \subseteq \gamma(v_1) \cap \ldots \cap \gamma(v_m) \cap S$.
Therefore,
\begin{align*}
\left| \gamma(x) \cap T_i \right| < \epsilon \cGreedy{m+1} \left| S \right| \leq \frac{\epsilon}{2} \left| T_i \right|,
\end{align*}
where the last inequality uses $|S| \leq \frac{1}{2\cGreedy{m+1}} |T_i|$ from Claim~1.
Assume that for every $j \neq i$, \eqref{eqdensityTj} fails.  Then
\begin{align*}
\left| \gamma(x) \cap S \right| 
&= \left| \gamma(x) \cap T_i \right| + \sum_{j \neq i} \left| \gamma(x) \cap T_j \right| 
\leq \frac{\epsilon}{2} \left| T_i \right| + 
   \sum_{j \neq i} (\cBoost + d(X,S)) \left| T_j \right|.
\end{align*}
Dividing through by $\left| \gamma(x) \cap S \right|$ we obtain
\begin{align*}
  1 \leq \frac{\epsilon}{2} \frac{\left| T_i \right|}{\left| S \right|} \frac{\left| S \right|}{\left| \gamma(x) \cap S \right|} +
   (\cBoost + d(X,S)) \left( 1 - \frac{\left| T_i \right|}{ \left| S \right|} \right) \frac{\left| S \right|}{\left| \gamma(x) \cap S \right|}.
\end{align*}
Because $\left| S \right|/\left| \gamma(x) \cap S \right| \leq \frac{1}{d(X,S)} \leq \frac{1}{\epsilon}$,
\begin{align} \label{vccolorbadeq}
1 \leq \frac{1}{2} \frac{\left| T_i \right|}{\left| S \right|} + \left(\frac{\cBoost}{\epsilon} + 1\right) \left( 1 - \frac{\left| T_i \right|}{\left| S \right|} \right).
\end{align}
Let $w = \left| T_i \right|/\left| S \right|$. The right hand side of the
above inequality is a weighted average of $\frac{1}{2}$ and
$(1+\frac{\cBoost}{\epsilon})$:
\begin{align*}
\frac{1}{2} w + \left(1 + \frac{\cBoost}{\epsilon} \right) (1 - w).
\end{align*}
Since $\frac{1}{2} < 1 + \frac{\cBoost}{\epsilon}$,
this will be maximized when $w$ is as small as possible.
By Claim~1, $w \geq \cRefineS$, and we have
\begin{align*}
\frac{1}{2} \cRefineS + \left(1 + \frac{\cBoost}{\epsilon} \right) (1 - \cRefineS) 
&< \frac{1}{2} \cRefineS + 1 + \frac{\cBoost}{\epsilon} - \cRefineS 
\leq 1 + \frac{\cBoost}{\epsilon} - \frac{1}{2} \cRefineS < 1.
\end{align*}
This implies that for any $w \geq \cRefineS$, the inequality in \eqref{vccolorbadeq} is false.
This contradiction shows that there must be some $j \neq i$ such that
$\left| \gamma(x) \cap T_j \right| / \left| T_j \right|$ is at least $\cBoost + d(X,S)$, which contradicts
that $E$ is contained in $B[Z]$.
\end{proof}
Thus $B[Z]$ is independent and the proof is complete.
\end{proof}

\section{Extremal results for critical hypergraphs}\label{extremal}

In this section, we prove Theorems~\ref{nearkextremal} and \ref{cycleiscritical}.  First, by
Lemma~\ref{notrpartite}, \C{2k+1}{r} is mono near $r$-partite.  Thus to complete the proof of
Theorem~\ref{cycleiscritical} we need only prove that \C{2k+1}{3} and \C{2k+1}{4} are stable with
respect to $T_3(n)$ and $T_4(n)$.  One tool we will use is the hypergraph removal lemma of Gowers,
Nagle, R\"{o}dl, and Skokan \cite{rrl-gowers07,rrl-nagle06,rrl-rodl04,Rodl06,rrl-tao06}.

\newtheorem{hypergraphremoval}[thmctr]{Theorem}
\begin{hypergraphremoval} \label{hypergraphremoval}
For every integer $r \geq 2$, $\epsilon > 0$, and $r$-uniform hypergraph $H$, there exists
a $\delta > 0$ such that any $r$-uniform hypergraph with at most $\delta n^{\left| V(H) \right|}$
copies of $H$ can be made $H$-free by removing at most $\epsilon n^r$ edges. \qed
\end{hypergraphremoval}

The second tool we will use is supersaturation, proved by Erd\H{o}s and Simonovits~\cite{rrl-erdos83}.
There are several equivalent formulations of supersaturation, the one we will use is the following.

\newtheorem{supersaturation}[thmctr]{Theorem}
\begin{supersaturation}  \label{supersaturation} \cite[Corollary 2]{rrl-erdos83}
Let $K^r_{t_1,\ldots,t_r}$ be the complete $r$-uniform, $r$-partite hypergraph with part sizes $t_1, \ldots, t_r$.  Let $t = \sum t_i$.
For every $\epsilon > 0$, there exists a $\delta = \delta(r,t,\epsilon)$ such that any $r$-uniform hypergraph with at least $\epsilon n^{r}$
edges contains at least $\delta n^t$ copies of $K^r_{t_1,\ldots,t_r}$. \qed
\end{supersaturation}

For any hypergraph $H$, let $H(t)$ denote the hypergraph obtained from $H$ by blowing up each vertex into an
independent set of size $t$.
An easy extension of supersaturation is the following (see Theorem~2.2 in the survey by Keevash~\cite{keevash11}).

\newtheorem{supersaturationblowup}[thmctr]{Corollary}
\begin{supersaturationblowup} \label{supersaturationblowup}
For every $r,t \geq 2$, $\epsilon > 0$, and $r$-uniform hypergraph $H$, there exists an $n_0$ such
that if $n \geq n_0$ and $G$ is an $n$-vertex, $r$-uniform hypergraph that contains at least $\epsilon n^{\left| V(H) \right|}$
copies of $H$, then $G$ contains a copy of $H(t)$. \qed
\end{supersaturationblowup}

Next, we will need stability results for $F_5$ and the book $B_{4,2}$, proved by
Keevash and \mubayi~\cite{keevash04} and Pikhurko~\cite{pikhurko08} respectively.
Let the \emph{book} $B_{r,m}$ be the $r$-uniform hypergraph with vertices $x_1, \ldots, x_{r-1}, y_1, \ldots, y_r$
and hyperedges $\left\{ x_1, \ldots, x_{r-1}, y_i \right\}$ for $1 \leq i \leq m$ and
$\left\{ y_1, \ldots, y_r \right\}$.  Note that $F_5 = B_{3,2}$.

\newtheorem{F5T3}[thmctr]{Theorem}
\begin{F5T3} \label{F5T3}~\cite{keevash04}
$F_5$ is stable with respect to $T_3(n)$. \qed
\end{F5T3}

\newtheorem{4unifbookstable}[thmctr]{Theorem}
\begin{4unifbookstable} \label{4unifbookstable}~\cite{pikhurko08}
$B_{4,2}$ is stable with respect to $T_4(n)$. \qed
\end{4unifbookstable}

\noindent The last piece of the proof of Theorem~\ref{cycleiscritical} is the following lemma.

\newtheorem{StabLem}[thmctr]{Lemma}
\begin{StabLem}\label{StabLem}
If $H$ is an $r$-uniform hypergraph that is stable with respect to $T_r(n)$
and $F$ is a non-$r$-partite subhypergraph of $H(t)$ for some $t$, then $F$ is also stable
with respect to $T_r(n)$.
\end{StabLem}

\begin{proof}
First, $\pi(F) \geq r!/r^r$.  Indeed, since $F$ is non-$r$-partite, $T_r(n)$ is an $F$-free
hypergraph.  To complete the proof that $F$ is stable with respect to $T_r(n)$, it is therefore
enough to prove that given $\epsilon > 0$, there exists a $\delta > 0$ such that if $G$ is an
$F$-free hypergraph with at least $t_r(n) - \delta n^r$ edges, then $G$ differs from $T_r(n)$ in at
most $\epsilon n^r$ edges.  This is enough since this implies that $\pi(F) \leq r!/r^r$ so $\pi(F) =
r!/r^r$.

Let $h$ denote the number of vertices in $H$ and let $\epsilon > 0$ be fixed.  We now show how to
define $\delta$. Since $H$ is stable with respect to $T_r(n)$, there exists an $\alpha \leq
\epsilon/2$ such that if $G'$ has at least $t_r(n) - 2\alpha n^r$ edges and contains no copy of $H$,
then $G'$ differs from $T_r(n)$ in at most $\epsilon n^r/2$ edges.  By
Theorem~\ref{hypergraphremoval}, there exists $\beta = \beta(\alpha)$ such that if there are at most
$\beta n^h$ copies of $H$ in $G$ then by deleting at most $\alpha n^r$ edges of $G$ we can remove
all copies of $H$.  Lastly, choose $\delta \ll \beta$.

Now, fix some $G$ that contains no copy of $F$ and has at least 
$t_r(n) - \delta n^r$ edges. Because $G$ contains no copy of $F$ it 
contains no copy of $H(t)$.  
Therefore, by Corollary~\ref{supersaturationblowup} there are at most $\beta n^h$ copies of $H$ in $G$.
By Theorem~\ref{hypergraphremoval}, we may therefore delete $\alpha n^r$ edges in order to find
a subhypergraph $G'$ of $G$ that contains no copy of $H$.  
Notice that $G'$ has at least $t_r(n) - (\delta + \alpha)n^r$ edges, and 
$(\delta + \alpha) < 2\alpha$, so $G'$ differs from $T_r(n)$ in at most $\epsilon n^r/2$ edges. 
Therefore, $G$ differs from $T_r(n)$ in at most $(\alpha + \epsilon/2)n^r$ edges, and 
$\alpha + \epsilon/2 < \epsilon$.
\end{proof}

It is easy to see that \C{2k+1}{r} is a non-$r$-partite subhypergraph of $B_{r,2}(k)$.  Thus
Theorem~\ref{F5T3} combined with Lemma~\ref{StabLem} shows that \C{2k+1}{3} is stable with respect
to $T_3(n)$ and similarly Theorem~\ref{4unifbookstable} combined with Lemma~\ref{StabLem} shows that
\C{2k+1}{4} is stable with respect to $T_4(n)$, which completes the proof of
Theorem~\ref{cycleiscritical}.

For $r \geq 5$, a result of Frankl and F\"{u}redi~\cite{frankl89} can be used to show that
\C{2k+1}{r} is not critical.

\newtheorem{c5notcritical}[thmctr]{Lemma}
\begin{c5notcritical} \label{c5notcritical}
For $r \geq 5$ and every $k \geq 1$, $\pi(C^r_{2k+1}) > \frac{r!}{r^r}$.
\end{c5notcritical}

\begin{proof}
Let $\mathcal{H}_n$ be the family of $r$-uniform hypergraphs $H$ on $n$ vertices 
that satisfy $|E_1\cap E_2|\le r-2$ whenever $E_1$ and $E_2$ are distinct edges of $H$.
It is easy to check that for any $t >0$ the blow-up $H(t)$ of $H$ is \C{2k+1}{r}-free.
Therefore, $ex(n,C^r_{2k+1})\ge \max_{H\in\mathcal{H}_{n/t}}\{ |H(t)|\}$. 
Frankl and F{\"u}redi~\cite{frankl89} showed that for $r\ge 7$,
\begin{align*}
\max_{H\in\mathcal{H}_{n/t}} \{|H(t)|\} > \frac{n^r}{r!}\frac{1}{\binom{r}{2}e^{1+1/(r-1)}}.
\end{align*}
Thus for $r \geq 7$, $\pi(C^r_{2k+1}) > \frac{r!}{r^r}$.

All that remains is the case when $r = 5$ or $6$.  Let $F$ be an $n$-vertex, $r$-uniform hypergraph
where no three edges $E_1,E_2,E_3$ satisfy $\left| E_1 \cap E_2 \right| = r-1$ and $E_1 \Delta E_2
\subseteq E_3$.  Frankl and F\"{u}redi~\cite{frankl89} proved that if $r = 5$ then for all such $F$
we have that $\left| E(F) \right| \leq \frac{6}{11^4}n^5$.  In addition, if $11$ divides $n$ there
exists a hypergraph $F$ achieving equality.  They also proved that if $r = 6$ then for all such $F$
we have that $\left| E(F) \right| \leq \frac{11}{12^5} n^6$; again, if $12$ divides $n$ then there
exists a hypergraph $F$ achieving equality.

Notice that if $H$ is the $r$-uniform hypergraph consisting of three hyperedges $E_1$, $E_2$, and
$E_3$ such that $|E_1 \cap E_2| = r-1$ and $E_1 \Delta E_2 \subseteq E_3$, then $C^r_{2k+1}$ is a
subhypergraph of a blowup of $H$.  Using supersaturation and an argument similar to that used in the
proof of Lemma~\ref{StabLem}, it follows that $$\pi(C^5_{2k+1}) = \frac{6!}{11^4} >
\frac{5!}{5^5}\text{ and }\pi(C^6_{2k+1}) = \frac{11\cdot 6!}{12^5} > \frac{6!}{6^6},$$ as claimed.
\end{proof}

\begin{proof}[Proof of Theorem~\ref{nearkextremal}]
Let $H$ be a critical $n$-vertex, $r$-uniform hypergraph. Suppose $H$ has $h$ vertices and assume that $E$ is the special edge of
a near $r$-partition that exhibits the fact that $H$ is critical, i.e., 
$E$ has at least $r-2$ vertices of degree one.
Suppose $G$ is an $H$-free, $r$-uniform, $n$-vertex hypergraph with $|G| \geq t_r(n)$.  
We would like to show that $G = T_r(n)$.  Partition the vertices of $G$ 
into parts $X_1, \ldots, X_r$ such that the number of edges with one vertex in
each $X_i$ is maximized.  Let $\epsilon_1 = (2r)^{-h}$, let $\epsilon_2 = \epsilon_1 / 8r^3$, let $\delta = \delta(r,h,\epsilon_2)$ from
Theorem~\ref{supersaturation}, and let $\epsilon < 2^{-2r} \epsilon_1 \epsilon_2 \delta$.
Organize $r$-sets of vertices into the following sets.
\begin{itemize}
\item Let $M$ be the set of $r$-sets with one vertex in each of $X_1, \ldots, X_r$ that are not edges of $G$
      (the missing cross-edges).
\item Let $B$ be the collection of edges of $G$ that have at least two vertices in some $X_i$ (the bad edges).
\item Let $G' = G - B + M$, so that $G'$ is a complete $r$-partite hypergraph.
\item Let $B_i = \left\{ W \in B : \left| W \cap X_i \right| \geq 2 \right\}$.
\end{itemize}
Since $B = \cup_i B_i$, there is some $B_i$ that has size at least $\frac{1}{r} \left| B \right|$.  Assume without 
loss of generality that $\left| B_1 \right| \geq \frac{1}{r} \left| B \right|$.  For $a \in X_1$, make the following definitions.
\begin{itemize}
\item $B_a = \left\{ W \in B_1 : a \in W \right\}$.
\item Let $C_{a,i}$ be the edges in $B_a$ that have exactly two vertices in $X_1$
  and exactly one vertex in each $X_j$ with $j \geq 2$ and $j \neq i$.
\item Let $D_a = B_a \setminus C_{a,2} \setminus \dots \setminus C_{a,r}$.
\end{itemize}

First, $\left| B \right| < \epsilon n^r$ because $G$ is stable with respect to $T_r(n)$.
Also, since $\left| G \right| \geq t_r(n)$, the number of $r$-sets in $M$ is at most the number of edges in $B$,
so $\left| M \right| \leq \left| B \right| < \epsilon n^r$.

In the rest of the proof, we will assume that $B$ is non-empty and then count the $r$-sets in $M$ in several different ways.
Our counting will imply that $\left| M \right| \geq \epsilon n^r$, and this contradiction will force $B = \emptyset$ and so
$G = T_r(n)$.  We will count $r$-sets in $M$ by counting embeddings of $H - E$ into $G'$ that also map $E$ to some element of $B$.
Since $G$ is $H$-free, each embedding must use at least one edge in $M$.
Let $\Phi$ be the collection of embeddings $\phi : V(H) \rightarrow V(G')$ of
$H - E$ into $G'$, by which we mean that $\phi$ is an injection and for all $F \in H$, $\phi(F) = \left\{ \phi(x) : x \in F \right\} \in G'$.
We say that $\phi \in \Phi$ is \emph{$W$-special} if $\phi(E) = W$
and \emph{$a$-avoiding} if $a \in V(G)$ and some degree one vertex in $E$ is mapped to $a$.
If $W \in B$ and $\phi$ is $W$-special, then $\phi$ must use at least one edge of $M$.  Call one of these edges the 
\emph{missing edge of $\phi$}.

\medskip
\noindent\emph{Claim 1:} For $\phi \in \Phi$ and $v \in V(H)$, there are at least $\frac{1}{2r} n$ embeddings $\phi' \in \Phi$
where $\phi(x) = \phi'(x)$ for $x \neq v$ and $\phi(v) \neq \phi'(v)$.

\begin{proof}
This follows easily because $G'$ is a complete $r$-partite hypergraph for which each class has size about $n/r$, and $\phi(v)$ can be replaced by any unused vertex
in the $X_i$ that contains $\phi(v)$.
\end{proof}

Fix some $W \in B$, and consider when there exists a $W$-special embedding of $H-E$.  Since $W \in B_i$ for some $i$,
let $w_1 \neq w_2 \in W \cap X_i$.  Then there exists an embedding of $H-E$ where $w_1$ and $w_2$ are used for the
non degree one vertices in the special edge of $H$.  Since the other vertices in the special edge have degree zero in $H-E$, the vertices
in the special edge can then be embedded to $W$.  
Thus for any $W \in B$, by Claim~1 there are at least $\epsilon_1 n^{h-r}$ $W$-special embeddings
of $H-E$, since we can vary any vertex of $H$ not in $W$.  The situation with $a$-avoiding is more complicated.
If $W \in C_{a,i}$, then the only choice of $w_1$ and $w_2$ that we are guaranteed to have 
are the two vertices in $W \cap X_1$, one of which is $a$.
Thus in a $W$-special embedding, the only way we can guarantee an embedding is by mapping a non-degree one vertex to $a$.
Therefore, only when $W \in D_a$ can we guarantee that
there exists at least $\epsilon_1 n^{h-r}$ $W$-special, $a$-avoiding embeddings of $H-E$.

\medskip
\noindent\emph{Claim 2:} For every $a \in X_1$, $\left| D_a \right| \leq \epsilon_2 n^{r-1}$.

\begin{proof}
Assume there exists some $a \in X_1$ with $\left| D_a \right| \geq \epsilon_2 n^{r-1}$.  We count $a$-avoiding, $W$-special embeddings of
$H - E$ into $G'$ where $W \in D_a$.  For each $W \in D_a$, we argued above that there are at least $\epsilon_1 n^{h-r}$ embeddings.
Since $\left| D_a \right| \geq \epsilon_2 n^{r-1}$, the number of $a$-avoiding embeddings that are $W$-special for some $W \in D_a$ is
at least $\epsilon_1 \epsilon_2 n^{r-1} \cdot n^{h-r} = \epsilon_1 \epsilon_2 n^{h-1}$.

Fix some $L \in M$.  We want to count the number of $a$-avoiding embeddings that are $W$-special for some $W \in D_a$ and
have missing edge $L$.   An upper bound on the number of such embeddings
will be the number of choices for $W$ times the number of choices for the $h - \left| W \cup L \right|$ vertices of $H$ mapped outside $W \cup L$.
Since all these embeddings are $a$-avoiding, $L$ cannot contain $a$.
For each $0 \leq \ell \leq r$, there exists at least $\binom{r}{\ell}$ choices for the
intersection between $L$ and $W$, at most $n^{r-\ell - 1}$ choices of $W \in D_a$ with $\left| W \cap L \right| = \ell$
(here it is crucial that $a \in W$ and $a \notin L$), and at most $n^{h-2r+\ell}$ choices for the vertices of $H$ not
in $W \cup L$.  Thus each $L \in M$ is in at most $2^{-r} n^{h-r-1}$ potential embeddings.  
Since there are at least $\epsilon_1 \epsilon_2 n^{h-1}$ embeddings, $M$ must have size
at least $2^{-r}\epsilon_1 \epsilon_2 n^r$, contradicting the choice of $\epsilon$.
\end{proof}

\noindent\emph{Claim 3:} For every $a \in X_1$ and every $2 \leq i \leq r$, $\left| C_{a,i} \right| \leq \epsilon_2 n^{r-1}$.

\begin{proof}
Assume there exists some $a$ and $i$ with $\left| C_{a,i} \right| \geq \epsilon_2 n^{r-1}$.  The proof is similar
to the proof of Claim~2, except now we cannot count $a$-avoiding embeddings.  In the previous claim,
we used the $a$-avoiding property to imply that the missing edge does not contain $a$.  In this proof, we will instead
guarantee that the missing edge cannot contain $a$ by only counting embeddings that map all neighbors of $\phi^{-1}(a)$ into $G$.

Let $v$ be one of the non degree one vertices in the special edge of $H$, and define
$H_v = \{ F \in H : \linebreak[1] v \in F, F \neq E \}$, that is all edges of $H$ containing $v$ that are not the special edge.
Let $Z_a = \left\{ F \in G \setminus B : a \in F \right\}$, that is all cross-edges of $G$ that contain $a$.
We now count embeddings $\phi \in \Phi$ that are $W$-special for some $W \in C_{a,i}$, map $v$ to $a$, and all edges of $H_v$ are
mapped to edges in $Z_a$.  For these embeddings, since edges in $H_v$ are mapped to edges in $Z_a \subseteq G$, the missing edge cannot contain $a$.

First, $\left| Z_a \right| \geq \left| C_{a,i} \right|$, because otherwise
we could move $a$ to $X_i$ and increase the number of edges across the partition and we chose the partition $X_1, \ldots, X_r$ to 
maximize the number of cross-edges. 
Let $H' = \left\{ F - v : F \in H_v \right\}$ and $Z' = \left\{ F - a : F \in Z_a \right\}$.
Then $H'$ and $Z'$ are $(r-1)$-uniform, $(r-1)$-partite hypergraphs, and $Z'$ has at least $\left| C_{a,i} \right| \geq \epsilon_2 n^{r-1}$ edges.
Let $t = \left| V(H') \right|$.  Then Theorem~\ref{supersaturation} shows that $Z'$ contains at least $\delta n^{t}$ copies of $H'$,
so there are at least $\epsilon_2 n^{r-1} \cdot \delta n^{t} \cdot \epsilon_1 n^{h-r-t} = \epsilon_1 \epsilon_2 \delta n^{h-1}$ embeddings
of $H-E$ that are $W$-special for some $W \in C_{a,i}$, map $v$ to $a$, and the edges in $H_v$ are embedded into $Z_a$.

Now fix $L \in M$, and consider how many of these embeddings have $L$ as their  missing edge. The computation is almost
the same as in the previous claim.  For each $\ell_1, \ell_2$, there are $\binom{r}{\ell_1}$ choices for $L \cap W$, there are $\binom{r}{\ell_2}$
choices for $L \cap \phi(H_v)$, there are $n^{r-1-\ell_1}$ choices for $W$ (here we use that $L$ does not contain $a$),
$n^{t - \ell_2}$ choices for $\phi(H_v)$, and $n^{h-2r - t +\ell_1 +\ell_2}$ choices for the other vertices of $H$.
Thus each $L$ is in at most $2^{2r} n^{h-r-1}$ potential embeddings.  Since
there are at least $\epsilon_1 \epsilon_2 \delta n^{h-1}$ embeddings, $M$ must have size at least $2^{-2r} \epsilon_1 \epsilon_2 \delta n^r$,
contradicting the choice of $\epsilon$.
\end{proof}

Claims 2 and 3 imply that $\left| B_a \right| < 2r\epsilon_2 n^{r-1}$ for each $a$.  Define
\begin{align*}
A = \left\{ a \in X_1 : d_M(a) \geq 2r^2 \epsilon_2 n^{r-1} \right\}.
\end{align*}
As in the proofs of the previous two claims, we would
like to count embeddings of $H - E$ to obtain a lower bound on $\left| M \right|$.
Once again, the main difficulty is controlling how the missing edge can intersect $W$.
If there were some $W$ with $W \cap A = \emptyset$,
then there would be few missing edges intersecting this $W$, which is how we will overcome this difficulty in
this part of the proof.

\medskip
\noindent\emph{Claim 4:} There exists some $W \in B_1$ with $W \cap A = \emptyset$.

\begin{proof}
Assume that every $W \in B_1$ contains an element of $A$.  Then $\sum_{a \in A} \left| B_a \right| \geq \left| B_1 \right|$.  Since
$\left| B_a \right| < 2r\epsilon_2 n^{r-1}$ for every $a$, we have the following contradiction.
\begin{align*}
2r \epsilon_2 n^{r-1} \left| A \right| > \sum_{a \in A} \left| B_a \right| \geq \left| B_1 \right| \geq \frac{1}{r} \left| B \right| \geq 
   \frac{1}{r} \left| M \right|  \geq \frac{1}{r} \sum_{a \in A} d_M(a) 
  \geq \frac{2r^2\epsilon_2}{r} n^{r-1} \left| A \right|.
\end{align*}
\end{proof}

We now complete the proof by counting the $W$-special embeddings whose missing edge does not intersect $W$.
There are at least $\epsilon_1 n^{h-r}$ embeddings that are $W$-special by Claim~1.  If at least half of these have missing edge
intersecting $W$, then $W$ would contain a vertex in $A$.  Thus there are at least $\frac{\epsilon_1}{2} n^{h-r}$ $W$-special embeddings
where the missing edge does not intersect $W$.  Each $L \in M$ is in at most $n^{h-2r}$ such potential embeddings, so
$M$ has at least $\frac{\epsilon_1}{2} n^r$ elements, contradicting the choice of $\epsilon$.
\end{proof}

\section{Chromatic threshold of $F_5$-free hypergraphs} \label{secF5}

\subsection{An upper bound on the chromatic threshold of $F_5$-free graphs}

In this section, we prove the upper bound in Theorem~\ref{mainF5thm}.  As in Section~\ref{secChromNearr},
we will give an upper bound on the chromatic threshold by first proving that large dimension
forces a copy of $F_5$, and then by applying Theorem~\ref{coloringthm}.
Let $(B,\gamma,F)$ be an $(\rb,\rg)$-uniform fiber bundle, and make the following definition.
A \emph{cut} in $(B,\gamma,F)$ is a pair $(X,S)$ such that
$X \subseteq V(B)$, $S \subseteq \binom{F}{\rg}$, and if $\gamma(x) \cap S \neq \emptyset$, then $x \in X$.  In other words,
the fibers that intersect $S$ come exclusively from $X$. 
A \emph{$k$-cut} is a cut $(X,S)$ with $\left| X \right| \leq k$.  The size of a $k$-cut is the
size of $\left| S \right|$.

We now sketch the proof of the upper bound in Theorem~\ref{mainF5thm}.  Let $G$ be
an $n$-vertex, $3$-uniform, $F_5$-free hypergraph with minimum degree at least $c\binom{n}{2}$.
Let $(G,\gamma,F)$ be the neighborhood bundle of $G$,
let $H = K_{q,q}$ for some large constant $q$ (see the definition of $q$ in the first line of the proof of Lemma~\ref{largedimimpliesU}), and assume $\vc_H(G,\gamma,F)$ is large.  We would like to find a copy of $F_5$ in
$G$.  We first use the fact that $\vc_H(G,\gamma,F)$ is large to find a set $U$ of vertices
of $G$ such that $G[U]$ has small strong independence number.  We then argue that
because the minimum degree is large, there must be some vertices $x,y$ such that
$N(x,y) = \left\{ z : xyz \in G \right\}$ has large intersection with $U$.  Next, we show that
since $N(x,y)$ has large intersection with $U$ and $G[U]$ has small strong independence number, there must
be an edge $E$ with at least two vertices in $N(x,y) \cap U$, which gives a copy of $F_5$.

The best upper bound on the chromatic threshold will come from the lowest required minimum degree needed in the
above proof.  The minimum degree is used above to prove that there exists some $x,y$ with
$N(x,y) \cap U$ large.  If we can find a large cut $(X,S)$ in $(G,\gamma,F)$ and we make $U$ large enough,
we could remove $X$ from $U$ while still maintaining all the useful properties of $U$.
Then for all $\{x,y\} \in S$, we know that $N(x,y) \cap (U - X) = \emptyset$.
Since there are now fewer pairs $\{x,y\}$ in $\binom{F}{2}$ with $N(x,y) \cap (U - X) \neq \emptyset$,
we can require a weaker lower bound
on the minimum degree of $G$ to find $\{x,y\}$ with $N(x,y) \cap U$ large.
In other words, the larger the cut of $(G,\gamma,S)$ we can find, the better upper bound on the chromatic threshold
we can prove.  This is encoded in the following theorem, which computes the relationship
between the minimum degree and the maximum size of a $k$-cut.

\newtheorem{f5largecutbound}[thmctr]{Theorem}
\begin{f5largecutbound} \label{f5largecutbound}
Let $0 \le c \le 1/5$, and fix an integer $k$ and a constant $c' > c$.
Then there exists a constant $L = L(c,c',k)$ such that the following holds.
Let $G$ be an $n$-vertex, $F_5$-free hypergraph with $\delta(G) \geq c' \binom{n}{2}$
and let $(G,\gamma,F)$ be the neighborhood bundle of $G$.  Assume $(G,\gamma,F)$ contains a $k$-cut of size
at least $(1-5c) \binom{n}{2}$.  Then $\chi(G) \leq L$.
\end{f5largecutbound}

Note that if $c = 1/5$, then $1-5c=0$ and so this theorem directly proves
an upper bound of $1/5$ on the chromatic threshold of $F_5$-free hypergraphs.
The first part of the proof of Theorem~\ref{f5largecutbound} is to find a set $U$ with
small strong independence number.

\newtheorem{largedimimpliesU}[thmctr]{Lemma}
\begin{largedimimpliesU} \label{largedimimpliesU}
Let $\epsilon > 0$ be fixed.  Then there exists constants $d = d(\epsilon)$ and $q = q(\epsilon)$ such
that the following holds.  Let $G$ be an $n$-vertex, $3$-uniform hypergraph and let $(G,\gamma,F)$ be the neighborhood bundle
of $G$.  Let $H = K_{q, q}$ and assume $\vc_H(G,\gamma,F) \geq d$.
Then there exists a vertex set $U \subseteq V(G)$ such that $\left| U \right|= 5d$ and the strong
independence number of $G[U]$ is at most $(1+\epsilon)d$.
\end{largedimimpliesU}

\begin{proof} Let $d = 100 + 100/\epsilon^2$ and $q = 3d + 2\cdot 3^d$.  Since $\vc_H(G,\gamma,F)
\geq d$, there exists a matching $E_1, \ldots, E_d$ such that for each $w_1 \in E_1, \ldots, w_d \in
E_d$ the section of $\left\{ w_1, \ldots, w_d \right\}$ contains a copy of $K_{q,q}$.  (See
Figure~\ref{nearkfigure} in Section~\ref{secChromNearr} for a picture of this structure.) Since $q =
3d + 2\cdot 3^d$, from each of these $3^d$ copies of $K_{q,q}$ we can pick a copy of $K_2$ such that
each $K_2$ is vertex disjoint from $E_1\cup\ldots\cup E_d$ and all these $3^d$ copies of $K_2$ are
vertex disjoint.  Now for $1 \leq i \leq d$, let $y_iz_i$ be a randomly chosen copy of $K_2$ (with
replacement), where each of the $3^d$ copies of $K_2$ are equally likely.  Let $Z = \left\{ y_1,
\ldots, y_{d}, z_1, \ldots, z_{d} \right\}$ and $U = Z \cup E_1 \cup \ldots \cup E_d$.  With
probability at most $\binom{d}{2} \frac{1}{3^d} < \frac{1}{4}$ some copy of $K_2$ is selected more than
once.  To finish the proof, we just need to show that with probability at most $1/4$, the
strong independence number of $G[U]$ is at least $(1+\epsilon)d$.  Indeed, in this case the union
bound shows that with probability at least $1/2$, $|U| = 5d$ and the strong independence number of
$G[U]$ is at most $(1+\epsilon)d$.

Notice that any strong independent set in $G[U]$ contains at most $d$ vertices from $E_1 \cup \ldots
\cup E_d$ and at most $d$ vertices from $Z$.  Thus any strong independent set in $G[U]$ with at
least $(1+\epsilon)d$ vertices must have at least $\epsilon d$ vertices in $E_1 \cup \ldots \cup
E_d$ and at least $\epsilon d$ vertices in $Z$.  We need to prove that this occurs with small
probability.

Let $x \in E_1 \cup \dots \cup E_d$ and $1 \leq i \leq d$.  We say that $\{y_i,z_i\}$ is \emph{built
from} $x$ if $\{y_i,z_i\}$ is an edge in a copy of $K_{q,q}$ which came from a section of $W$ where
$x \in W$.  That is, say $x \in E_j$.  Each section picks one of the three vertices of $E_j$ and if
the section picks $x$ and $\{y_i,z_i\}$ is the edge chosen from the copy of $K_{q,q}$ chosen from
this section, then we say that $\{y_i,z_i\}$ is \emph{built from} $x$.  For $x \in E_1 \cup \dots
\cup E_d$ and $1 \leq i \leq d$, let $A_{x,i}$ be the following event:
\begin{align*}
  A_{x,i}: \{y_i,z_i\} \text{ is built from $x$}.
\end{align*}
First, $\mathbb{P}[A_{x,i}] = 1/3$.  Indeed, say $x \in E_j$ and note that there are $3^d$ copies of
$K_2$ in total (there are three choices from each of $E_1,\dots,E_d$ for the section) and there are
$3^{d-1}$ copies of $K_2$ built from $x$.  Therefore, when randomly picking copies of $K_2$, the
probability that $\{y_i, z_i\}$ is built from $x$ is exactly $1/3$.

Let $\mathcal{S} = \{ S \subseteq E_1 \cup \dots \cup E_d : |S| = \epsilon d \text{ and $S$ has at
most one vertex in each $E_i$}\}$.  We claim that the events $A_{x,i}$ for $x \in S$ are mutually
independent for every $S \in \mathcal{S}$.  Indeed, fix some $Q \subseteq S$.  Then
\begin{align*}
  \mathbb{P}\left[ \bigwedge_{x \in Q} A_{x,i} \right] = \frac{3^{d-|Q|}}{3^d} = \left( \frac{1}{3}
  \right)^{|Q|}
\end{align*}
since there are $3^{d-|Q|}$ of the copies of $K_2$ built from $x$ for
$x \in Q$ and built on any of three vertices in the edges $E_j$ that do not contain a vertex of
$Q$ (recall that $S$ has at most one vertex in each $E_j$).  Thus $\mathbb{P}[\wedge_{x\in Q}
A_{x,i}] = \prod_{x\in Q} \mathbb{P}[A_{x,i}]$ so that for every $S \in \mathcal{S}$ the events
$A_{x,i}$ for $x \in S$ are mutually independent.  Therefore,
\begin{align*}
  \mathbb{P} \left[ \bigwedge_{x \in S} \overline{A_{x,i}} \right] = \left( \frac{2}{3}
  \right)^{|S|}.
\end{align*}
Let $B_{S,i}$ be the event
\begin{align*}
  B_{S,i}: \text{ no edge of $G$ contains a vertex of $S$ and both $y_i$ and $z_i$}.
\end{align*}
If $B_{S,i}$ holds, then for every $x \in S$ it is the case that the event $A_{x,i}$ fails since if
$A_{x,i}$ holds then $\{y_i,z_i,x\} \in E(G)$.  Thus
\begin{align*} \label{f5:probAsi}
  \mathbb{P}[B_{S,i}] \leq \mathbb{P} \left[ \bigwedge_{x \in S} \overline{A_{x,i}} \right] = \left( \frac{2}{3}
  \right)^{|S|}.
\end{align*}
For each $T \subseteq [d]$ with $|T| = \epsilon d$, let $B_{S,T}$ be the conjunction of the
events $B_{S,i}$ for all $i \in T$.  The events $B_{S,i}$ are mutually independent for $i \in T$
since the copies of $K_2$ were selected with replacement, so that $\mathbb{P}[B_{S,T}] \leq (2/3)^{|S||T|}$.
Let $X_{S,T}$ be the indicator random variable for the event
$B_{S,T}$ and let $X$ be the sum of all indicator random variables over all $S \in \mathcal{S}$ and
all $T \subseteq [d]$ with $|T| = \epsilon d$.  We now have $\binom{d}{\epsilon d}$ choices for $T$
and $3^{\epsilon d} \binom{d}{\epsilon d}$ choices for $S$ so that
\begin{align*}
  \mathbb{E}[X] = \sum X_{S,T} \leq 3^{\epsilon d} \binom{d}{\epsilon d}^2 \left( \frac{2}{3} \right)^{\epsilon^2 d^2} 
\leq \left( 3 \left( \frac{e}{\epsilon} \right)^2  \left( \frac{2}{3} \right)^{\epsilon d}
\right)^{\epsilon d} < \frac{1}{4}.
\end{align*}
By Markov's inequality, the probability that $X \geq 1$ is at most $1/4$ so that with probability at
most $1/4$, some $B_{S,T}$ holds.  If $W$ is a strong independent set in $G[U]$ with $|W| \geq
(1+\epsilon)d$, then $|W \cap Z| \geq \epsilon d$ and $|W \cap (E_1\cup\dots\cup E_d)| \geq \epsilon
d$.  Also, $W$ uses at most one vertex from each pair in $Z$ so that there exists $T\subseteq[d]$ of
size $\epsilon n$ such that for $i \in T$ we have that either $y_i$ or $z_i$ is in $W$.  Since $W$
uses at most one vertex from each $E_i$, there exists $S \subseteq W \cap (E_1\cup\dots\cup E_d)$
with $|S| = \epsilon d$ and $S \in \mathcal{S}$.  Since $W$ is a strong independent set the event
$B_{S,T}$ holds.  Therefore, the probability some $B_{S,T}$ holds is an upper bound for the
probability the strong independence number of $G[U]$ is at least $(1+\epsilon)d$.  Since the
probability that some $B_{S,T}$ holds is at most $1/4$, the proof is complete.
\end{proof}

We can now prove Theorem~\ref{f5largecutbound}.

\begin{proof}[Proof of Theorem~\ref{f5largecutbound}]
Pick $\epsilon$ so that $c' = (1+2\epsilon)c$ and let $d = d(\epsilon)$ and $q = q(\epsilon)$ be given by 
Lemma~\ref{largedimimpliesU}, and also assume that $d$ is large enough so that $5d\epsilon > k(1+2\epsilon)$.
Suppose that if $H = K_{q,q}$ then $\dim_H(G,\gamma,F) \leq d$.  Then by Theorem~\ref{coloringthm}, there exists
constants $K_1 = K_1(\epsilon,d,H)$ and $K_2 = K_2(\epsilon,d,H)$ (note that $K_1$ and $K_2$ depend only on $c,c',k$)
such that either $\left| F \right| < K_1$ or $\chi(G) < K_2$.  Since $\left| F \right| = \left| V(G) \right|$,
this implies that $\chi(G) < \max\{K_1, K_2\}$.

We can therefore assume that $\dim_H(G,\gamma,F) \geq d$.  By Lemma~\ref{largedimimpliesU}, there exists a set
$U \subseteq V(G)$ such that $\left| U \right| = 5d$ and the strong independence number of $G[U]$ is at most $(1+\epsilon)d$.
Let $(X,S)$ be a $k$-cut of size at least $(1-5c)\binom{n}{2}$.
Let $G'$ be the bipartite graph with partite sets $A = U \setminus X$ and $B=\binom{V(G)}{2}\setminus S$ where $\left\{ u ,\left\{ v,w\right\}\right\}$ is an edge in $G'$ if and only if $\left\{ u , v,w\right\}$ is an edge in $G$. $\left| A \right|\ge 5d - \left| X \right|$, so $G'$ contains at least $(5d - \left| X \right|) \delta(G)$ edges. $ \left| B \right|=\binom{n}{2} - \left| S \right|$, so there is some $x \neq y$ such that $d_{G'}\left(\{x,y\}\right)$ is at least

\begin{align*}
\frac{(5d - \left| X \right|) \delta(G)}{\binom{n}{2} - \left| S \right|} 
\geq \frac{(5d - k) (1+2\epsilon)c \binom{n}{2}}{5c \binom{n}{2}}
= \frac{(5d - k) (1+2\epsilon)}{5}
> (1+\epsilon)d.
\end{align*}
This implies that there is some $x,y$ with $\left| N(x,y) \cap U \right| > (1+\epsilon)d$.
Since the strong independence number of $G[U]$ is at most $(1+\epsilon)d$,
there exists some edge $E$ with two vertices in $N(x,y)$.
Then $x,y$ together with $E$ form a copy of $F_5$ in $G$.  This contradiction completes the proof.
\end{proof}

\subsection{Finding a large cut in an $F_5$-free hypergraph}

In order to use Theorem~\ref{f5largecutbound} to prove the upper bound in Theorem~\ref{mainF5thm}, we now need to show the existence of a large cut.
Note that in Theorem~\ref{f5largecutbound} the bound on the chromatic number depends
on $k$ but there are no other restrictions on $k$.  Thus to prove an upper bound on the chromatic
threshold of a $F_5$-free graph $G$, one can pick any fixed integer $k$ and ask what is the size of the largest $k$-cut.
In the following lemma, we set $k = 5$ and prove that if $\delta(G)\ge c'\binom{n}{2}$ with $c'>c$, then there exist a $5$-cut of $G$ of size approximately
$4c^2 \binom{n}{2}$.  Solving $4c^2 = 1-5c$ gives $c = (\sqrt{41}-5)/8$, the bound in Theorem~\ref{mainF5thm}.

We suspect that the bound on the chromatic threshold of $F_5$-free hypergraphs can
be improved by finding a larger cut, perhaps by increasing $k$.  In order to achieve a bound of $c = 6/49$,
we would need to find a cut of size $s\binom{n}{2}$ with $s = 1-5c=539/36 c^2 \approx 15c^2$.

\newtheorem{F5implieslargecut}[thmctr]{Lemma}
\begin{F5implieslargecut} \label{F5implieslargecut}
Let $0 < c < c'$ be fixed.  There exists a constant $n_0 = n_0(c,c')$ such that for all $n > n_0$ the following holds.
Let $G$ be an $n$-vertex, $3$-uniform, $F_5$-free hypergraph with $\delta(G) \geq c'\binom{n}{2}$. Let
$(G,\gamma,F)$ be the neighborhood bundle of $G$. Then $(G,\gamma,F)$ has a 
$5$-cut of size at least $4 \binom{c(n-1)}{2}$.
\end{F5implieslargecut}

\noindent Combining Theorem~\ref{f5largecutbound} with Lemma~\ref{F5implieslargecut}, we can prove
Theorem~\ref{mainF5thm}.

\begin{proof}[Proof of Theorem~\ref{mainF5thm}]
Let $c = (\sqrt{41}-5)/8$, let $c' > c$ be fixed, and let $G$ be any $n$-vertex, $3$-uniform, $F_5$-free graph
with minimum degree at least $c' \binom{n}{2}$. Let $(G,\gamma,F)$ be the neighborhood bundle of $G$.
Let $b = (c' + c)/2$ so that $c' > b > c$.
Then by Lemma~\ref{F5implieslargecut}, either $\left| V(G) \right|$ is bounded or $(G,\gamma,F)$ contains a $5$-cut
of size at least $4\binom{b(n-1)}{2}$.  Since $b > c$, if $n$ is large enough this is at least $4c^2 \binom{n}{2}$.
Notice that $4c^2 = 1-5c$, so Theorem~\ref{f5largecutbound} implies that the
chromatic number of $G$ is bounded.
\end{proof}

\noindent The first step in the proof of Lemma~\ref{F5implieslargecut} is the following lemma.

\newtheorem{goodpairlemma}[thmctr]{Lemma}
\begin{goodpairlemma} \label{goodpairlemma}
In a graph $G$, we call a non-edge $uv \notin E(G)$ \emph{good} if $N(u)\cap N(v)\ne\emptyset$.
If $G$ is a triangle-free graph with $n$ vertices and $m$ edges, then $G$ has at least $m-n/2$ good
non-edges.
\end{goodpairlemma}

\begin{proof} 
We prove this by induction on $n$. It is obviously true for $n=1$ and $n=2$. Now
assume $n>2$. If some component of $G$ is not regular, then there exist vertices $u,v$ in that
component such that $u\in N(v)$ and $d(u)<d(v)$. Then $G-u$ has $n-1$ vertices and $m-d(u)$ edges.
By induction, $G-u$ has at least $m-d(u)-\frac{n-1}{2}$ good non-edges. For any vertex $w\in N(v)-u$,
$uw$ is a good non-edge, so $G$ has at least $m-d(u)-\frac{n-1}{2}+d(v)-1\ge m-n/2$ good
non-edges. If all components of $G$ are regular, then pick one component $K$. Assume $K$ is
$r$-regular, choose a vertex $v$ in $K$, and let $N_{2}(v)=\left\{u: \mbox{ there exists a $P_3$
connecting $u$ and $v$}\right\}$. If $|N_2(v)|\geq r$, then by the induction hypothesis $G-v$ has at least
$m-r-\frac{n-1}{2}$ good non-edges, and since for any vertex $u\in N_2(v)$ it is the case that $uv$ is a good non-edge,
$G$ has at least $m-r-\frac{n-1}{2}+|N_2(v)|\geq m-n/2$ good non-edges. If $|N_2(v)|< r$, then since $K$ is
triangle-free and $r$-regular, $K$ is the complete bipartite graph $K_{r,r}$, which has $r^2$ edges and $r^2-r$ good non-edges. Now $G-K$
has $n-2r$ vertices and $m-r^2$ edges, so by induction it has $m-r^2-(n-2r)/2$ good non-edges. Then $G$
has $m-r^2-(n-2r)/2+r^2-r=m-n/2$ good non-edges.  
\end{proof}  

\begin{proof}[Proof of Lemma~\ref{F5implieslargecut}]
We examine the copies of $F_4$ in $G$ where $F_4$ is the
hypergraph with vertex set $\left\{ 1,2,3,4 \right\}$ and edges $\{1,2,3\}$, $\{1,2,4\}$, 
and $\{2,3,4\}$.

\medskip 
\noindent\emph{Case 1}: There exists a vertex $v$ of $G$ such that $v$ is not contained in
any copy of $F_4$.  Consider $L = \gamma(v)[V(G) - v]$, which is a triangle-free graph with $n-1$
vertices and at least $c \binom{n}{2}$ edges. By Lemma~\ref{goodpairlemma}, $L$ has at least 
$c\binom{n}{2}-\frac{n-1}{2}$ good non-edges. Let $X=\emptyset$ and $S$ be the set of these good non-edges.
We claim that $(X,S)$ is a cut in $(G,\gamma,F)$. Suppose for contradiction that there exists some
$x\in V(G)$ and $\left\{ u,w \right\}\in S$ such that $\left\{ u,w,x \right\}\in G$. Pick a vertex
$y$ from $N_L(u)\cap N_L(w)$. Then $u,v,w,x,y$ form a copy of $F_5$ in $G$, which is a contradiction. 
  
\medskip
\noindent\emph{Case 2}: Every vertex of $G$ is contained in some copy of $F_4$.
Pick some $U \subseteq V(G)$ such that $G[U] = F_4$, let
$U=\left\{ u_1,u_2,u_3,u_4\right\}$, and let $G'=\cup_{i=1}^{4} \gamma(u_i)$.
Consider $\gamma(u_i) \cap \gamma(u_j)$ for $i \neq j$.  If $\gamma(u_i) \cap \gamma(u_j)$
contains a matching of size two, then $G$ contains a copy of $F_5$.  Say $ab, cd \in \gamma(u_i)
\cap \gamma(u_j)$ with $a,b,c,d$ distinct.  Then since $G[U] = F_4$, there is some edge
$E = \left\{ u_i, u_j, w \right\} \in G$.  If $w \neq a$ and $w \neq b$, then $a,b,u_i,u_j,w$ form a copy
of $F_5$ and if $w = a$ or $w = b$, then $c,d,u_i,u_j,w$ form a copy of $F_5$.  Thus
$\gamma(u_i) \cap \gamma(u_j)$ is a star so has at most $n$ elements.  Since each
$\gamma(x)$ has size at least $c' \binom{n}{2}$, $G'$ has at least
$4c' \binom{n}{2} - \binom{4}{2} n > 4c \binom{n}{2}$ edges if $n$ is large enough.

Then $G'$ has $n$ vertices and at least $4c\binom{n}{2}$ edges, so there exist a vertex $v$ whose degree
in $G'$ is at least $4c(n-1)$. Let $N$ denote the neighborhood of $v$ in $G'$ and let
$N_1,\dots,N_{4}$ be a partition of $N$ such that for every $1\le i\le 4$ and every vertex
$w\in N_i, vw\in\gamma(u_i)$. Let $X=U\cup \{v\}$ and $S = \bigcup_{i=1}^{4} \binom{N_{i}}{2}$,
so that $|X|=5$ and $|S|\geq 4\binom{|N|/4}{2}= 4\binom{c(n-1)}{2}$. We claim that
$(X,S)$ is a cut in $(G,\gamma,F)$. Suppose for contradiction that there exists some $z\notin X$ such
that $\gamma(z)\cap S\neq \emptyset$. Pick $\{x,y\}\in \gamma(z)\cap S$, then $\{x,y\} \subseteq
N_i$ for some $1 \leq i \leq 4$. Now $v,u_i,x,y,z$ form a copy of $F_5$, which is a contradiction.

From these two cases we can see that $(G,\gamma,F)$ has a $5$-cut of size at least 
\[\min\left\{c \binom{n}{2}-\frac{n-1}{2}, 4\binom{c(n-1)}{2}\right\}.\] 
Because $G$ is $F_5$-free, it follows that $c\le 2/9$ and therefore $\min\left\{c \binom{n}{2}-\frac{n-1}{2},4\binom{c(n-1)}{2}\right\}=4\binom{c(n-1)}{2}$.
\end{proof}

\subsection{A construction for the lower bound} \label{f5lowersection}

To prove a lower bound on the chromatic threshold of the family of $F_5$-free hypergraphs,
we need to construct an infinite sequence of $F_5$-free hypergraphs with large chromatic number and large minimum degree.
Our construction is inspired by a construction by Hajnal~\cite{cr-erdos73} of a dense triangle-free graph with high chromatic number.
Hajnal's key idea was to use the Kneser graph to obtain large chromatic number. 
The Kneser graph $\kneser(n,k)$ has vertex set $\binom{[n]}{k}$, and two vertices $F_1, F_2$ form an edge if and only if
$F_1 \cap F_2=\emptyset$.  We use an extension of Kneser graphs to hypergraphs.
Alon, Frankl, and Lov\'{a}sz~\cite{kn-alon86} considered the Kneser hypergraph $\kneser^r(n, k)$, which is the $r$-uniform hypergraph
with vertex set $\binom{[n]}{k}$, and $r$ vertices $F_1, \ldots, F_r$ form an edge
if and only if $F_i \cap F_j = \emptyset$ for $i \neq j$. They gave a lower bound
on the chromatic number of $\kneser^r(n,k)$ as follows.

\newtheorem{kneserlower}[thmctr]{Theorem}
\begin{kneserlower} \label{kneserlower}
If $n \geq (t-1)(r-1) + rk$, then $\chi(\kneser^r(n, k)) \geq t$.
\end{kneserlower}

\noindent We first show that $\kneser^r(n,k)$ is $F_5$-free for $n < 4k$.

\newtheorem{kneserf5free}[thmctr]{Lemma}
\begin{kneserf5free} \label{kneserf5free}
If $n < 4k$, then $\kneser^3(n,k)$ is $F_5$-free.
\end{kneserf5free}

\begin{proof}
Say $\left\{ a,b,c \right\}$, $\left\{ a,b,d \right\}$ and $\left\{ c,d,e \right\}$ are edges in $\kneser^3(n,k)$.  Then
by definition $a$, $b$, $c$, and $d$ are four disjoint $k$-sets in $[n]$, which is impossible because $n < 4k$.
\end{proof}


\begin{proof}[Proof of the lower bound in Theorem~\ref{mainF5thm}]
Fix $t \geq 2$ and $\epsilon > 0$.  Pick $k \geq 2t$ and $n = 3k + 2(t-1)$ and note that $n < 4k$.  
By Theorem~\ref{kneserlower},
$\kneser^3(n,k)$ has chromatic number at least $t$ and by Lemma~\ref{kneserf5free} is 
$F_5$-free.  For integers $u$, $v$, and $w$ where $n$ divides $u$, let $U$, $V$ and $W$ be 
disjoint vertex
sets of size $u$, $v$, and $w$ respectively.  Partition $U$ into $U_1, \ldots, U_n$ such that 
$\left| U_i \right| = \frac{u}{n}$ for each $i$.
Let $H$ be the hypergraph with vertex set $V(\kneser^3(n,k)) \cup U \cup V \cup W$ 
and the following edges.
\begin{itemize}
\item For $\left\{ S_1, S_2, S_3 \right\} \in \kneser^3(n,k)$, make 
      $\left\{ S_1, S_2, S_3 \right\}$ an edge of $H$.
\item For $S \in V(\kneser^3(n,k))$, $x \in U_i$ with $1 \leq i \leq n$, and $y \in V$, 
      make $\left\{ S,x,y \right\}$ an edge of $H$ if $i \in S$.
\item For $x \in U$, $y \in V$, and $z \in W$, make $\left\{ x,y,z \right\}$ an 
edge of $H$.
\end{itemize}
\noindent Notice that $H$ has chromatic number at least $t$ because $\kneser^3(n,k)$ is a subhypergraph.  

\medskip
\noindent\emph{Claim 1:} $H$ contains no subgraph isomorphic to $F_5$.
\begin{proof}
Suppose $\{a, b, c\}, \{a,b,d\}$ and $\{c,d,e\}$ are the hyperedges of a copy of $F_5$ in $H$. 
Notice that the hypergraph induced by $U, V, V(\kneser^3(n,k))\cup W$ is $3$-partite, 
apart from those edges within $\kneser^3(n,k)$. Note that a $3$-uniform, $3$-partite hypergraph is $F_5$-free, therefore any copy of $F_5$ must contain 
an edge from $\kneser^3(n,k)$.  If that edge is $\{a,b,c\}$ then $d$ must also be 
contained in $V(\kneser^3(n,k))$.  But then $c$ and $d$ are both in $V(\kneser^3(n,k))$, which 
means $e$ must be as well.  Because $\kneser^3(n,k)$ is $F_5$-free, this is a contradiction. 
Similarly, $\{a,b,d\} \subsetneq V(\kneser^3(n,k))$.  Therefore, 
$\{c,d,e\} \subseteq V(\kneser^3(n,k))$, and without loss of generality $b\in U$ and $a \in V$.  
Because $\{a,b,c\}$ and $\{a,b,d\}$ are edges, $b$ must be in both $c$ and $d$, which contradicts 
the fact that $\{c,d,e\}$ is an edge of $\kneser^3(n,k)$.
\end{proof}

\medskip
\noindent\emph{Claim 2:} The minimum degree of $H$ is at least $(1-\epsilon) \frac{6}{49} \binom{\left| V(H) \right|}{2}$ if
$\left| V(H) \right|$ is large enough.

\begin{proof}
Vertices in $\kneser^3(n,k)$ have degree at least $k\frac{u}{n}v = \frac{kuv}{3k + 2(t-1)}$.  
Since $t$ is fixed, we can choose $k$ large enough that 
vertices in $\kneser^3(n,k)$ have degree at least $(1-\epsilon/2) uv/3$.  Vertices in $A$ 
have degree at least $vw$, vertices in $B$ have degree at least $uw$,
and vertices in $C$ have degree at least $uv$.  Thus the minimum degree of $H$ 
is at least $\min\left\{ (1-\epsilon/2)\frac{uv}{3}, uw, vw \right\}$.
Choose $u$, $v$, and $w$ so that $\frac{uv}{3} = uw = vw$, we obtain that $u = v$ 
and $w = v/3$ and the minimum degree is at least $(1-\epsilon/2) u^2/3$.
The number of vertices is $u + v + w + \binom{n}{k} = \frac{7}{3} u + \binom{n}{k}$.  Since 
$u^2/3 \approx 6/49 \binom{7u/3}{2}$, we can choose $u$ large enough so that
the minimum degree of $H$ is at least 
$(1-\epsilon) \frac{6}{49} \binom{\left| V(H) \right|}{2}$.
\end{proof}

We have proved that for every fixed $t \geq 2$ and every $\epsilon > 0$, there is a constant $N_0$ such that for $N > N_0$ there exists
an $N$-vertex, $3$-uniform, $F_5$-free hypergraph with chromatic number at least $t$ and
minimum degree at least $(1-\epsilon)\frac{6}{49} \binom{\left| V(H) \right|}{2}$.  By the definition of chromatic threshold,
this implies that the chromatic threshold of the family of
$F_5$-free hypergraphs is at least $\frac{6}{49}$.
\end{proof}

\section{Generalized Kneser hypergraphs} \label{secnewKneser}

In Section~\ref{f5lowersection}, we used a generalization of the Kneser graph to hypergraphs to give a lower bound on the
chromatic threshold of the family of $F_5$-free hypergraphs.  In Section~\ref{secopenproblems}, we will use similar constructions
to give lower bounds on the chromatic threshold of the family of $A$-free hypergraphs, for several other hypergraphs $A$.  For some of
these constructions, we will need a more general variant of the Kneser hypergraph, which we explore in this section.

Sarkaria~\cite{kn-sar90} considered the generalized Kneser hypergraph $\kneser^r_s (n,k)$, which is the $r$-uniform hypergraph with vertex
set $\binom{[n]}{k}$, in which $r$ vertices $F_1,\dots, F_r$ form an edge if and only if no element of $[n]$ is contained in more than $s$ of
them. Note that the Kneser hypergraph  $\kneser^r(n, k)$ is $\kneser^r_1(n, k)$. Sarkaria~\cite{kn-sar90} and Ziegler~\cite{kn-zie02}
gave lower bounds on the chromatic number of $\kneser^r_s (n,k)$, but Lange and Ziegler~\cite{kn-lan-zie07} showed that the lower bounds
obtained by Sarkaria and Ziegler apply only if one allow the edges of $\kneser^r_s (n,k)$ to have repeated vertices. We conjecture that for $\kneser^r_s (n,k)$,
a statement similar to Theorem~\ref{kneserlower} is true. 

\newtheorem{newkneserconj}[thmctr]{Conjecture}
\begin{newkneserconj}\label{newkneserconj}
There exists $T(r,s,t)$ such that if $n\geq T(r,s,t)+rk/s$, then
$\chi\left(\kneser^r_s(n,k)\right)\geq t$. \qed
\end{newkneserconj}

The following much weaker statement is sufficient for our purposes. The proof is similar to an argument of Szemer{\'e}di which
appears in a paper of Erd\H{o}s and Simonovits~\cite{cr-erdos73},
and the proof of Claim $1$ is motivated by an argument of Kleitman~\cite{Kleitman66}.

\newtheorem{newKneser}[thmctr]{Theorem}
\begin{newKneser} \label{newKneser}
Let $c>0$; then for any integers $r,t$, there exists $K_0=K_0(c,r,t)$ such that if $k\geq K_0$, $s=r-1$, and $n=(r/s+c)k$, then $\chi\left(\kneser^r_s(n,k)\right)> t$.
\end{newKneser}

Before we prove this theorem, we need two definitions. A family $\F$ of subsets of $[n]$ is \emph{monotone decreasing} if $F\in \F$ and $F'\subseteq F$ imply $F'\in \F$. Similarly, it is \emph{monotone increasing} if $F\in \F$ and $F\subseteq F'$ imply $F'\in \F$.

\begin{proof}[Proof of Theorem~\ref{newKneser}]
Fix an integer $t$.  We would like to prove that if $k$ is large enough then it is impossible
to $t$-color $\kneser^r_s(n,k)$.  So let $k$ be some integer and assume
$\kneser^r_s(n,k)$ can be $t$-colored. Then the $k$-subsets of $[n]$ can be divided into $t$ families, $\F_1,\dots, \F_t$, such that $F_1\cap\cdots\cap F_{r}\ne\emptyset$ for all distinct $F_1,\dots, F_r\in \F_i,1\le i\le t$. For $1\leq i\leq t$, let $\F_i^*=\{A:A\subseteq[n],\exists F\in \F_i\textrm{ such that }F\subseteq A\}$. Then $\F_1^*,\dots,\F_t^*$ are monotone increasing families of subsets of $[n]$. Let $w=s/r$;
since $s= r - 1$, $w = 1 - 1 /r$.  For a family $\F$ of subsets of $[n]$, define the weighted size $W[\F]$ of $\F$ by
$$W[\F]=\sum_{F\in\F}w^{|F|}(1-w)^{n-|F|}.$$

\noindent\emph{Claim 1:} For $1\leq \ell \leq t, W[\cup_{i=1}^{\ell}\F_i^*]\le 1-1/r^{\ell}$.

\begin{proof}
 We prove this by induction on $\ell$. For $\ell=1$, Frankl and Tokushige~\cite{kn-frankl03} showed that for a family $\F$ of subsets of $[n]$, if $F_1\cap\cdots\cap F_{r}\ne\emptyset$ for all distinct $F_1,\dots, F_r\in \F$, then
$W[\F]\le w = 1-1/r$. Now assume that the statement is true for $\ell$. Let $U = \cup_{i=1}^\ell\F_i^*$ and $L = \overline{\F_{\ell+1}^*}$. Then $W[U]\le 1-1/r^{\ell}$, $U$ is a monotone increasing family of subsets of $[n]$, and $L$ is a monotone decreasing family of subsets of $[n]$. By the FKG Inequality, 
$$W[U\cap L]\le W[U]W[L].$$
Then 
\begin{displaymath}
\begin{array}{ccccc}W[\cup_{i=1}^{\ell+1}\F_i^*] & = & W[U\cap L]+W[\F_{\ell+1}^*] & \leq & W[U]W[L]+W[\F_{\ell+1}^*] \\
 & \leq & (1-1/r^\ell)W[L]+W[\F_{\ell+1}^*] & = & 1-(1-W[\F_{\ell+1}^*])/r^{\ell}.
 \end{array}
\end{displaymath}

Since $W[\F_{\ell+1}^*]\le w = 1-1/r$, we have $1-(1-W[\F_{\ell+1}^*])/r^{\ell}\le 1-1/r^{\ell+1}$, so $W[\cup_{i=1}^{\ell+1}\F_i^*]\le 1-1/r^{\ell+1}$.
\end{proof}

Now we know that $W[\cup_{i=1}^t\F_i^*]\le 1-1/r^t$, so $W[\overline{\cup_{i=1}^t\F_i^*}]\ge 1/r^t$. We also know that $\overline{\cup_{i=1}^t\F_i^*}$ is the family of subsets of $[n]$ whose size is less than $k = n/(r/s+c)$, so 
$$W[\overline{\cup_{i=1}^t\F_i^*}]=\sum_{i<\frac{n}{r/s+c}}\binom{n}{i}w^i(1-w)^{n-i}.$$

Since $wn=\frac{n}{r/s}>\frac{n}{r/s+c}$, by Chernoff's inequality we have
$$ \sum_{i<\frac{n}{r/s+c}}\binom{n}{i}w^i(1-w)^{n-i}\le
  e^{-\left(\frac{c}{r/s+c}\right)^2\frac{sn}{2r}}=e^{-\frac{c^2s}{2(r/s+c)r}k}.$$
Then if $k$ is large and $t$ is fixed, $W[\overline{\cup_{i=1}^t\F_i^*}]\leq e^{-\frac{c^2s}{2(r/s+c)r}k} < 1/r^t$ which contradicts Claim~1.
This contradiction implies that for any fixed $t$, there is no choice of $K_0$ such that for all $k > K_0$ it is possible
to $t$-color $\kneser^r_s(n,k)$.  This completes the proof.
\end{proof}

For an $r$-uniform hypergraph $A$, we want to construct an infinite sequence of $A$-free hypergraphs with $\kneser^r(n,k)$ or $\kneser^r_{r-1}(n,k)$ as a subhypergraph.  This will imply that these $A$-free hypergraphs have large chromatic number, but we must first show that for any integer $k$ and for some choice of $n = n(k)$ one of $\kneser^r(n,k)$, $\kneser^r_{r-1}(n,k)$ is $A$-free.
We now show that $\kneser^3_2(n,k)$ is $T_5$-free and $S(7)$-free under some conditions on $n$ and $k$.
Here $T_5$ is a $3$-uniform hypergraph with vertices $v_1,v_2,v_3,v_4,v_5$ and edges 
$\{v_1,v_2,v_3\}$, $\{v_1,v_4,v_5\}$, $\{v_2,v_4,v_5\}$, $\{v_3,v_4,v_5\}$, and $S(7)$ denotes the Fano plane
(the $S$ stands for Steiner Triple System.)

\newtheorem{newkneserT5free}[thmctr]{Lemma}
\begin{newkneserT5free} \label{newkneserT5free}
If $n < (3/2+1/4)k$, then $\kneser^3_2(n,k)$ is $T_5$-free.
\end{newkneserT5free}
\begin{proof}
If $n<3k/2$, then $\kneser^3_2(n,k)$ has no edge and of course is $T_5$-free.  Assume $n=(3/2+\epsilon)k$ with $0\le \epsilon<1/4$, and suppose $T_5$ is a subhypergraph of $\kneser^3_2(n,k)$. Since $\{v_1,v_4,v_5\}$, $\{v_2,v_4,v_5\}$, $\{v_3,v_4,v_5\}$ are 
edges of $T_5$, the vertices $v_1$, $v_2$, and $v_3$ all lie in $\overline{v_4 \cap v_5}$. Because $\left|\overline{v_4 \cap v_5}\right|\le 2n-2k = (1+2\epsilon)k<3k/2$, by the pigeonhole principle, $v_1 \cap v_2 \cap v_3\ne\emptyset$, which means $\{v_1,v_2,v_3\}$ is not an edge, a contradiction.
\end{proof}

\newtheorem{newkneserS7free}[thmctr]{Lemma}
\begin{newkneserS7free} \label{newkneserS7free}
If $n < (3/2+1/10)k$, then $\kneser^3_2(n,k)$ is $S(7)$-free.
\end{newkneserS7free}
\begin{proof}
  Just as in the proof of Lemma~\ref{newkneserT5free}, assume $n=(3/2+\epsilon)k$ with $0\le \epsilon<1/10$ and suppose $S(7)$ is a subhypergraph of $\kneser^3_2(n,k)$. Let $A$ be a vertex in a copy of $S(7)$ in $\kneser^3_2(n,k)$ and let $\{A,B,C\},\left\{ A,D,E \right\}, \left\{ A,F,G \right\}$ be its incident edges in the copy of $S(7)$. Then $B\cap C, D\cap E, F\cap G \subseteq \overline{A}$.  Since $\left|\overline{A} \right|=(1/2+\epsilon)k$, $\left|B\cap C\right|, \left|D\cap E\right|, \left|F\cap G\right|\ge (1/2-\epsilon)k$. Then since $3(1/2-\epsilon)>2(1/2+\epsilon)$, the pigeonhole principle implies that $B\cap C\cap D\cap E\cap F\cap G\ne \emptyset$. Now the copy of $S(7)$ cannot have an edge not containing $A$, a contradiction. 
\end{proof}

We will use Lemma~\ref{newkneserS7free} in Subsection~\ref{subsecS7} to provide a lower bound on the chromatic threshold of the family of $S(7)$-free hypergraphs.  
Similarly, we will use Lemma~\ref{newkneserT5free} in Subsection~\ref{subsecT5} to provide a lower bound on the chromatic threshold of the family of $T_5$-free hypergraphs.

\section{Open Problems and Partial Results} \label{secopenproblems}


Many open problems remain; for most $3$-uniform hypergraphs $A$ the
chromatic threshold for the family of $A$-free hypergraphs is unknown.
Interesting hypergraphs to study are those for which we know the extremal number, $ex(n,A)$,
and we will examine a few of those here along with partial results and conjectures.
We conjecture that most of the lower bounds given by the constructions
in this section are tight.

\begin{figure}
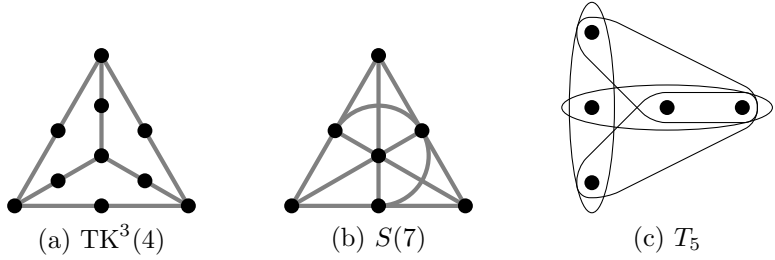

\begin{center}
\subfloat[$\tk{4}{3}$]{\tkfourfigure}
\hspace{0.3in}
\subfloat[$S(7)$]{\ssevenfigure}
\hspace{0.3in}
\subfloat[$T_5$]{\tfivefigure}
\end{center}
\caption{Assorted Hypergraphs.}
\end{figure}

\subsection{$\tkf{s}{r}$-free hypergraphs}\label{subsecTK}

For $s > r$, recall that $\tkf{s}{r}$ is the family of $r$-uniform hypergraphs such that there exists a
set $S$ of $s$ vertices where each pair of vertices from $S$ are contained together in some edge.  The set $S$
is called the set of \emph{core vertices} of the hypergraph.
Recall also that $T_{r,s}(n)$ is the complete $n$-vertex, $r$-uniform, $s$-partite hypergraph with
part sizes as equal as possible.

The last author~\cite{rt-mubayi06-2} showed that if $s > r$ then
$\ex(n, \tkf{s}{r}) = \left| T_{r,s-1}(n) \right|$ and
$\ex(n, \linebreak[1] \tk{s}{r}) = (1 + o(1)) \left| T_{r,s}(n) \right|$.  Recently,
Pikhurko~\cite{rt-pikhurko} has shown that for large $n$ and $s > r$,
$\ex(n,\tk{s}{r}) \linebreak[1] = \left| T_{r,s-1}(n) \right|$
and that $T_{r,s-1}(n)$ is the unique extremal example.
Because $F_5$ is a member of $\tkf{4}{3}$ it follows that the chromatic threshold of
$\tkf{4}{3}$-free hypergraphs is at most $(\sqrt{41}-5)/8$.
The following simple variation on the construction from Section~\ref{f5lowersection} provides a lower
bound of $18/361$ for both $\tk{4}{3}$-free
and $\tkf{4}{3}$-free hypergraphs.

\newtheorem{tkflowerlemma}[thmctr]{Proposition}
\begin{tkflowerlemma} \label{tkflowerlemma}
The chromatic threshold of $\tkf{4}{3}$-free hypergraphs is at least $\frac{18}{361}$.
\end{tkflowerlemma}

\begin{proof}
The proof is very similar to the proof in Section~\ref{f5lowersection}, we only sketch it here.
Choose $k,n,u,v,w,U, V, W$ as in the proof of the lower bound of Theorem~\ref{mainF5thm} in Section~\ref{f5lowersection};
that is $k,n,u,v,w$ are integers with $n \ll u,v,w$ and $U,V,W$ are disjoint sets of vertices of size $u,v,w$
respectively.  Divide $U$ into $U_1, \ldots, U_n$ so
that $\left| U_i \right| = u/n$ and divide $V$ into $V_1, \ldots, V_n$ such that
$\left|V_i \right| = v/n$.
Let $H$ be the hypergraph formed by taking $\kneser^3(n,k)$ and adding the complete
$3$-partite hypergraph on $U,V,W$ and the following edges.
For $S \in V(\kneser^3(n,k))$ and $x \in U_i$ and $y \in V_j$, make $\left\{ S, x, y \right\}$
an edge if $i,j \in S$.  The minimum
degree is maximized when $u=v$ and $w = u/9$, which gives minimum degree
approximately $uv/9 \approx \frac{18}{361}\cdot \binom{N}{2}$, where $N = u+v+w+\binom{n}{k}$
is the number of vertices in the hypergraphs.

Let $F$ be any hypergraph in $\tkf{4}{3}$ and assume that $F$
is a subhypergraph of $H$ in which $c_1, c_2, c_3, c_4$ are the four core vertices.
Because any $3$-partite hypergraph is $\tkf{4}{3}$-free,
it is easy to see that some edge of $F$ must lie in $\kneser^3(n,k)$, and so there must
be at least two core vertices in $\kneser^3(n,k)$.
If $c_1, c_2 \in \kneser^3(n,k)$ and $c_3 \in U \cup V$
then $c_3$ is in either $U_i$ or $V_i$ for some $i$.
But then $i \in c_1 \cap c_2$ (recall that vertices in $\kneser^3(n,k)$ are $k$-sets) which
contradicts the fact that $c_1$ and $c_2$ are contained together
in some edge of $\kneser^3(n,k)$.  Thus  all four core vertices must be in
$\kneser^3(n,k)$, which is not possible because $n < 4k$.
\end{proof}

This gives lower bounds on the chromatic thresholds of
$\tk{4}{3}$-free and $\tkf{4}{3}$-free hypergraphs and leads to the following questions.

\newtheorem{tkfquestion}[thmctr]{Question}
\begin{tkfquestion}
What is the chromatic threshold for $\tk{4}{3}$-free hypergraphs? It is between $18/361$
and $2/9$.  What is the chromatic
threshold for $\tkf{4}{3}$-free hypergraphs?  It has the same lower bound as for $\tk{4}{3}$-free
hypergraphs, and because $F_5 \in \tkf{4}{3}$ the upper bound is $(\sqrt{41}-5)/8$.
\end{tkfquestion}

\noindent A similar construction provides a $\tkf{s}{3}$-free hypergraph for any $s \geq 5$.
We have not optimized the values.

\newtheorem{tkflowerlemma2}[thmctr]{Lemma}
\begin{tkflowerlemma2} \label{tkflowerlemma2}
When $s \geq 5$, the chromatic threshold of $\tkf{s}{3}$-free hypergraphs is at least
$\frac{(s-2)(s-3)(s-4)^2}{(s^2 - 13)^2} = 1- \frac{13}{s} + O(\frac{1}{s^2})$.
\end{tkflowerlemma2}

\begin{proof}
Fix $t \geq 2$, $k \geq 2t$, and let $n = 3k + 2(t-1)$.  Notice that $n < 4k$.
By Theorem~\ref{kneserlower}, the chromatic number of $\kneser^3(n, k)$ is therefore
at least $t$. Fix $N \gg \binom{n}{k}$.

Partition $N$ vertices into one part of size $u$ and $s-2$ parts of size $x$, for
some $u$ that is divisible by $n$.  Include as an edge each triple that has at
most one vertex in each part.  Further partition the part of size $u$ into $n$ sets,
$U_1, \dots, U_n$, each of size $u/n$. From the remaining $s-2$ parts of size $x$,
choose two and designate them $W_1, W_2$; label the remaining $s-4$ parts
$V_1, \dots, V_{s-4}$.
Let $H$ be the $3$-uniform hypergraph
formed by taking the disjoint union of $\kneser^3(n,k)$ and the above complete $(s-1)$-partite
hypergraph, and adding the following edges.
If $S \in V(\kneser^3(n,k))$, $v \in V_i$, and $v' \in V_j$ for $i \neq j$,
add the edge $\{S, v, v'\}$.
If $S \in V(\kneser^3(n,k))$ and $u \in U_i$ and $v \in V_j$ then add the edge
$\{S, u, v\}$ if and only if $i \in S$.

\noindent Notice that $H$ has chromatic number at least $t$, and that
$V(H) = N + \binom{n}{k}$.

\medskip
\noindent\emph{Claim 1:} $H$ contains no element of $\tkf{s}{3}$
as a subgraph.
\begin{proof} Suppose there is such a
subgraph; then at least one core vertex must be contained in $ V(\kneser^3(n,k))$, because an
$(s-1)$-partite
graph is $\tkf{3}{s}$-free.  In that case, no core vertex can be in $W_1 \cup W_2$ because
there is no edge that contains a vertex from $W_1 \cup W_2$ as well as a vertex from $ V(\kneser^3(n,k))$.
There must therefore be at least $3$ core vertices in $ V(\kneser^3(n,k))$, which means that two of
them must appear in an edge contained within $ V(\kneser^3(n,k))$.  Suppose they are $S_1, S_2$.
If another core vertex is in
$U$, say $u \in U_i$, then there must be an edge of $H$ containing $u$ and $S_1$,
and there must be an edge containing $u$ and $S_2$.  This implies that $i \in S_1 \cap S_2$,
which contradicts the fact that $S_1$ and $S_2$ appear together in an edge of $\kneser^3(n,k)$.

All core vertices must therefore be in $ V(\kneser^3(n,k)) \cup V$, which means that
there must be at least four of them
in $ V(\kneser^3(n,k))$.  Because each pair of those four core vertices must appear together
in an edge, and that edge must
be in $\kneser^3(n,k)$, those four sets must be pairwise disjoint.  This is impossible
because $n < 4k$.
\end{proof}

\noindent The minimum degree of this graph is approximately
\[
\min \left\{
\frac{1}{3}(s-4)ax + \binom{s-4}{2}x^2, \binom{s-2}{2}x^2, (s-3)ax + \binom{s-3}{2}x^2
\right\}.
\]
Notice that a vertex in $W_1 \cup W_2$ has degree strictly less than a vertex in $\kneser^3(n,k)$,
and so they do not enter into the above computation.
This minimum is largest when $ u = \frac{3(2s-7)x}{s-4}$, which implies that
$x = \left(\frac{s-4}{s^2 - 13}\right) N$.  The minimum degree of $H$ is
then
\[
\frac{(s-2)(s-3)}{2}\cdot\frac{(s-4)^2}{(s^2 - 13)^2}N^2 = \left(1 - \frac{13}{x} + O\left(\frac{1}{s^2}\right) \right)\frac{N^2}{2}.
\]
\end{proof}

The construction in Lemma~\ref{tkflowerlemma2} has one part of ``type'' $U$ (which is partitioned into $n$
sets), $s-4$ parts of ``type'' $V$ (which are not partitioned, and whose vertices appear in edges
that intersect $K$), and two parts of ``type'' $W$ (which are not partitioned and have no vertices
that appear in edges intersecting $K$).
Using this strategy, one can generate similar constructions for $\tk{s}{r}$;
the above proof applies whenever there are $x$ parts of type $U$, $s - (r+1)$ parts of type
$V$, and $y$ parts of type $W$, where $x + y = r$ and $s-(r+1) + x \geq r-1$.  This last
condition is needed for the edges intersecting $K$.

\newtheorem{tkfsquestion}[thmctr]{Question}
\begin{tkfsquestion}
What is the chromatic threshold for $\tk{s}{3}$-free hypergraphs for $s > 3$? It is between
$\frac{(s-2)(s-3)(s-4)^2}{(s^2 - 13)^2} = 1 - \frac{13}{s} + O\left(\frac{1}{s^2}\right)$ and
$\left(1-\frac{1}{s-1}\right)\left(1-\frac{2}{s-1}\right) = 1 - \frac{3}{s-1} + \frac{2}{(s-1)^2}$.
The upper bound comes from  $T_{r,s-1}(n)$.
\end{tkfsquestion}

\subsection{$S(7)$-free hypergraphs}\label{subsecS7}

Next, consider the Fano plane $S(7)$.
de Caen and F\"{u}redi~\cite{tdf-decaen00} showed that
$\ex(n,S(7)) = (\frac{3}{4}+o(1))\binom{n}{3}$.
The extremal hypergraph for $S(7)$, proven to be extremal by F\"{u}redi and
Simonovits~\cite{tdf-furedi05}
and also by Keevash and Sudakov~\cite{tdf-keevash05}, is the hypergraph
formed by taking two almost equal vertex sets $U$ and $V$ and taking
all edges that have at least one vertex in each of $U$ and $V$.  We can modify
the hypergraph from Section~\ref{f5lowersection}
to obtain a lower bound on the chromatic threshold of $S(7)$-free hypergraphs.

\newtheorem{fanolowerlemma}[thmctr]{Proposition}
\begin{fanolowerlemma} \label{fanolowerlemma}
The chromatic threshold of $S(7)$-free hypergraphs is at least
$9/17$.
\end{fanolowerlemma}

\begin{proof}
Fix $t\geq 2$ and $0<\epsilon \ll 1$.  Then by Lemma~\ref{kneserlower} there exists $k$ sufficiently large
that if $n = (3 + \epsilon)k$ then $\kneser^3(n,k)$ has chromatic number at least $t$.  Fix
such a $k$, and fix $N \gg \binom{n}{k}$.

Partition $N$ vertices into two sets, $U$ and $V$, with $|U| = 9N/17$ and $|V| = 8N/17$.
Further partition $U$ into $n$ parts, $U_1,\ldots,U_n$, each of size $|U|/n$.
Include as an edge each triple that has at least one vertex in each of $U$, $V$.  Let $H$ be
the hypergraph formed by taking the disjoint union of this hypergraph and $\kneser^3(n,k)$ and
adding the following edges. For $u \in U_i$, $u' \in U_j$, and $X\in V(\kneser^3(n,k))$ include
$\{X,u,u'\}$ as an edge if $i,j \in X$ (recall that vertices in $\kneser^3(n,k)$ are subsets of $[n]$).
Let $K = V(\kneser^3(n,k))$. Notice that $H$ has chromatic number at least $t$, and that
$V(H) = N + \binom{n}{k}$.

\medskip
\noindent\emph{Claim 1:} $H$ contains no subhypergraph isomorphic to $S(7)$.
\begin{proof}
First notice that $\kneser^3(n,k)$ is $S(7)$-free because every pair of vertices in $S(7)$
are in an edge, which would require there to be $7$ pairwise-disjoint $k$-subsets of $[n]$.
Because $n = (3 +\epsilon)k$, this would be a contradiction.
It is easy to see, by considering the partition $U,(K\cup V)$, that if $H$ contains a copy of
$S(7)$ then it must involve an edge from $H[K]$ (otherwise the extremal $S(7)$-free hypergraph also
contains a copy of $S(7)$). Call this edge $\{A, B, C\}$.

There are four vertices in $S(7) \setminus \{A, B, C\}$, and at least one must be outside
$K$. No more than one can be in $V$ because there is no edge with one vertex in $K$ and
two in $V$.  No more than one can be in $U$ otherwise one of $A \cap B$, $A \cap C$, $B \cap C$
is non-empty, which contradicts the assumption that $\{A, B, C\}$ is an edge of $H[K]$.
Therefore, there must be either $5$ or $6$ vertices of $S(7)$ in $K$. Suppose $v$ is a vertex of $S(7)$
that is outside of $K$.  Then $v$ appears in three edges that overlap only at $v$, say $\{v, S_1, S_2\}$,
$\{v, S_3, S_4\}$, and $\{v, S_5, S_6\}$.  At least one of these edges must contain two vertices from $K$,
but there is no such edge in $H$.
\end{proof}

The minimum degree of $H$ is at least
\[ \min\left\{|U||V|+\binom{|U|/3}{2}, |U||V|+\binom{|U|}{2}, |U||V|+\binom{|V|}{2} \right\} = \frac{9}{34}N^2-\frac{3}{34}N. \]
\end{proof}

\newtheorem{fanoquestion}[thmctr]{Question}
\begin{fanoquestion}
What is the chromatic threshold of $S(7)$-free hypergraphs?  It is at least
$9/17$ and at most $3/4$, where the upper bound is from the extremal hypergraph
of $S(7)$.
\end{fanoquestion}

\subsection{$T_5$-free hypergraphs}\label{subsecT5}

Recall that the $3$-uniform hypergraph $T_5$ has vertices $A, B, C, D, E$ and edges $\{A, B, C\}$,
$\{A, D, E\}$, $\{B, D, E\}$, and $\{C, D, E\}$.

Let $B^3(n)$ be the $3$-uniform hypergraph with the most edges among all $n$-vertex $3$-graphs
whose vertex set can be partitioned into $X_1, X_2$ such that each edge contains
exactly one vertex from $X_2$. F\"{u}redi, Pikhurko, and Simonovits~\cite{fur-pik-sim}
proved that for $n$ sufficiently large the extremal $T_5$-free hypergraph is $B^3(n)$. It follows that
the chromatic threshold for the family of $T_5$-free hypergraphs is at most $4/9$.

\newtheorem{T5lowerlemma}[thmctr]{Proposition}
\begin{T5lowerlemma} \label{T5lowerlemma}
The chromatic threshold of $T_5$-free hypergraphs is at least $16/49$.
\end{T5lowerlemma}

\begin{proof}
Fix $t\geq 2$ and $0<\epsilon \ll 1$.  Then by Lemma~\ref{kneserlower} there exists $k$ sufficiently large
that if $n = (3/2 + \epsilon)k$ then $\kneser^3_2(n,k)$ has chromatic number at least $t$.  Fix
such a $k$, and fix $N \gg \binom{n}{k}$.

Partition $N$ vertices into two parts, $U$ and $V$, with $|U| = 4N/7$ and $|V| = 3N/7$. Further
partition $U$ into $n$ parts, $U_1,\dots,U_n$, each of size $|U|/n$.  Include as an edge any triple
with two vertices in $U$ and one in $V$.  Let $H$ be the hypergraph formed by taking the disjoint
union of this graph and $\kneser^3_2(n,k)$ and including the following edges.  If
$X \in V(\kneser^3_2(n,k))$ and $u\in U_i$ and $v \in V$ then let $\{u,v,X\}$ be an edge if
$i \in X$ (recall that vertices of $\kneser^3_2(n,k)$ are subsets of $[n]$).  Let
$K = V(\kneser^3_2(n,k))$.  Notice that $H$ has chromatic number at least $t$, and that
$V(H) = N + \binom{n}{k}$.

\medskip
\noindent\emph{Claim 1:} $T_5$ is not a subhypergraph of $H$.
\begin{proof}
Let $H'$ be the hypergraph obtained from $H$ by deleting all edges contained in $K$, and let
$X_1 = K \cup U$ and $X_2 = V$.  It is now easy to see that $H'$ is a subhypergraph of the
extremal $T_5$-free hypergraph; if $H$ contains a copy of $T_5$ it must therefore involve an
edge from $K$.  If that edge is $\{A,D,E\}$ (see the labelling of $T_5$ above) then because
$\{B,D,E\}$ and $\{C,D,E\}$ are edges of $T_5$ it must be the case that both of $B,C$ are in $K$, but
by Lemma~\ref{newkneserT5free} $K$ does not span a copy of $T_5$.
Similarly, neither $\{B,D,E\}$ nor $\{C,D,E\}$ can be contained in $K$.

We may therefore assume that $\{A,B,C\}$ is contained in $K$.  Because $\{A,D,E\}$ is an edge, and
by Lemma~\ref{newkneserT5free}, at least one of $D,E$ is in $U$.  Suppose
that $D \in U_i$; then because $\{A, D, E\}$, $\{B, D,E\}$, and $\{C,D,E\}$ are all edges
of $T_5$ it must be the case that $i \in A \cap B \cap C$.  This contradicts the assumption that
$\{A,B,C\}$ is an edge.
\end{proof}

The minimum degree of $H$ is at least
\[ \min\left\{\frac{2|U||V|}{3}, |U||V|, \binom{|U|}{2} \right\} = \frac{8}{49}N^2-\frac{2}{7}N. \]
\end{proof}

\subsection{Co-chromatic thresholds}

There is another possibility when generalizing the definition of chromatic
threshold from graphs to hypergraphs:
we can use the co-degree instead of the degree.
Recall that if $H$ is an $r$-uniform hypergraph and
$\left\{ x_1, \ldots, x_{r-1} \right\} \subseteq V(H)$,
 then the \emph{co-degree $d(x_1, \ldots, x_{r-1})$ of $x_1,\ldots,x_{r-1}$} is
$\left| \left\{ z : \left\{ x_1, \ldots, x_{r_1}, z \right\} \in H \right\} \right|$.
Let $F$ be a family of $r$-uniform hypergraphs.
The \emph{co-chromatic threshold} of $F$ is the infimum
of the values $c \geq 0$ such that the subfamily of $F$ consisting of
hypergraphs $H$ with
minimum co-degree
at least $c \left| V(H) \right|$ has bounded chromatic number.
More generally, the \emph{$k$-degree $d(x_1, \ldots, x_k)$ of $x_1, \ldots, x_k$} is
$\left| \left\{ \left\{ z_{k+1}, \ldots, z_r \right\} : \left\{ x_1, \ldots, x_k, z_{k+1}, \ldots, z_r \right\} \in H \right\} \right|$
and we can define the $k$-chromatic threshold similarly.
Given a hypergraph $H$ and subsets $U,V,W$ of $V(H)$, we say that an edge $\{u,v,w\}$
is of type $UVW$ if $u\in U, v\in V$ and $w\in W$.

The co-chromatic thresholds of $F_5$-free hypergraphs and $\tk{4}{3}$-free
hypergraphs are trivially zero because if the minimum co-degree
of $H$ is at least $10$ then $H$
contains a copy of $\tk{4}{3}$ and a copy of $F_5$.
For the Fano plane, the last author proved~\cite{mubayi-fano} that for every $\epsilon>0$ there exists $n_0$ such that any $3$-uniform hypergraph with $n>n_0$ vertices and
minimum co-degree greater than $(1/2 + \epsilon)n$  contains a copy of $S(7)$.
In 2009, Keevash~\cite{keevash-fano} improved this by proving that any $3$-uniform hypergraph with
minimum co-degree greater than $n/2$ contains a copy of $S(7)$ for $n$ sufficiently large.
Notice that the lower bound construction for the chromatic threshold described above has
non-zero minimum co-degree but the co-degree depends on the parameter $t$.  We can modify the construction
to prove a better lower bound on the co-chromatic threshold of $S(7)$-free hypergraphs.

\newtheorem{Fanonewextremal}[thmctr]{Proposition}
\begin{Fanonewextremal}\label{Fanonewextremal}
The co-chromatic threshold of $S(7)$-free hypergraphs is at least $2/5$.
\end{Fanonewextremal}
\begin{proof}
Fix $t\geq2$ and $0<\epsilon \ll 1$.  Then by Lemma~\ref{newKneser} there exists $k$ large enough that if
$n = (3/2 + \epsilon)k$ then $\kneser^3_2(n,k)$ has chromatic number at least $t$.  Fix
$N \gg \binom{n}{k}$.

Partition $N$ vertices into two parts, $U$ and $V$, of size $\frac{3N}{5}$ and $\frac{2N}{5}$
respectively. Include as an edge any triple with at least
one vertex in each part.  Further partition $U$ into $n$ sets, $U_1, \dots, U_n$, each of size
$|U|/n$.  Let $H$ be the hypergraph formed by taking the disjoint union of this hypergraph with
$\kneser^3_2(n,k)$ and including the following edges.  Include any edge of type $KUV$, where
$K = V(\kneser^3_2(n,k))$.  For any $X, Y \in K$, if $|X \cap Y| < k - 4\epsilon k$ then
include every edge of the form $\{X,Y,u\}$ where $u \in U_i$ for some $i \in X \cup Y$.
If $|X \cap Y| \geq k - 4\epsilon k$ then include every edge of
the form $\{X, Y, u\}$ where $u \in U_i$ for some $i \in X \cap Y$.
Notice that $H$ has chromatic number at least $t$ and that
$V(H) = N + \binom{n}{k}$.

\medskip 

\noindent\emph{Claim 1:} The above hypergraph contains no subgraph isomorphic to $S(7)$.
\begin{proof}
First notice that the complete bipartite $3$-uniform hypergraph contains no copy of $S(7)$.  Therefore,
by considering the partition $U, V \cup K$, we can see that any copy of $S(7)$ must contain an edge
induced by $K$.  Call this edge $\{A,B,C\}$.
It also follows from Lemma~\ref{newkneserS7free} that there is no
copy of $S(7)$ completely contained in $K$.

\medskip
\noindent\emph{Claim 1a:} Any copy of $S(7)$ intersects $U$ (or $V$) in at most one vertex.
\begin{proof}
Notice that for any edge $e$ in $S(7)$, every other edge intersects $e$
in at exactly one vertex; therefore for any copy of $S(7)$ in $H$ every edge contains one of $A, B, C$.
If there were two vertices of $S(7)$ in $U$ (or in $V$) then the edge of $S(7)$ joining them would be unable
to intersect $A, B$, or $C$.
\end{proof}

\noindent\emph{Claim 1b:} Any copy of $S(7)$ contains no vertex from $V$.
\begin{proof}
Suppose for contradiction a copy of $S(7)$ contains some vertex from $V$; then by Claim~1a it
intersects $V$ in exactly one vertex.  Every vertex of $S(7)$ is contained
in three edges, but because there is at most one vertex from $U$ involved in the copy of $S(7)$
there can be only one edge that contains the vertex from $V$.
\end{proof}

Any copy of $S(7)$ must therefore have
exactly six vertices in $K$ and exactly one vertex in $U$.  Suppose they are $A, B, C, D, E, F \in K$
and $G \in U_i$.  Suppose also that the edges of $S(7)$ induced by $K$ are
\[\{A,B,C\}, \{A, E, F\}, \{C, D, E\}, \{B, D, F\}.\]

\noindent\emph{Claim 1c:} If $\{S_1, S_2, S_3\}$ is an edge in $K$ then
$|S_i \cap S_j| \leq k/2 + \epsilon k$ for all $i \neq j$.
\begin{proof}
This follows from the definition of the hypergraph on $K$:
\[
k = |S_1| \leq n - |S_2 \cap S_3| = (3/2 + \epsilon)k - |S_2 \cap S_3|\mbox{, so } |S_2 \cap S_3| \leq k/2 + \epsilon k,
\]
and the claim follows through symmetry.
\end{proof}

\noindent\emph{Claim 1d:} The following intersections all have size at least
$2k - 4\epsilon k$: $A \cap D, B \cap E, C \cap F$.

\begin{proof}
We will prove that $|A \cap D| \geq 2k - 4\epsilon k$; the rest follow through symmetry.
Because $\{B,D, F\}$ is an edge, $D \subseteq (\overline{B}\cap F) \cup (B \cap \overline{F})
\cup (\overline{B} \cap \overline{F})$.
Also, because $\{A, B, C\}$ is an edge,
$|\overline{A} \cap \overline{B}| = |\overline{A}| - |\overline{A}\cap B| \leq (k/2 + \epsilon k)
- (k/2 - \epsilon k) = 2\epsilon k$.
Similarly, because $ \{A, E, F\}$ is an edge, $|\overline{A}\cap\overline{F}| \leq 2\epsilon k$. Therefore,
\[
|D \cap \overline{A}| \leq
					|\overline{A} \cap \overline{B} \cap F| + |\overline{A} \cap B \cap \overline{F}| +
					|\overline{A} \cap \overline{B} \cap \overline{F}|
					\leq |\overline{A} \cap \overline{B}| + |\overline{A} \cap \overline{F}|
					\leq 4\epsilon k,
\]
and so $|D \cap A | \geq |D| - 4\epsilon k = k - 4\epsilon k$.
\end{proof}

It follows from Claim~1d that $S(7)$ cannot be a subgraph of $H$.  Otherwise, the edges
$\{A, D, u\}, \{B, E , u\}, \{C, F, u\}$ would all appear, and by the definition of $H$,
because the intersections mentioned in Claim~1d are large, it follows that
$i \in (A \cap D) \cap (B \cap E) \cap (C \cap F)$.
In that case, however, $A \cap B \cap C$ is not empty and so $\{A, B, C\}$ is not an edge.
\end{proof}

It remains only to compute the minimum degree of $H$.
Vertices $S_1, S_2 \in K$ have co-degree at least $\frac{k-4\epsilon k}{n}|U|$ if
$|S_1 \cap S_2| \geq k-4\epsilon k$ and at least $\frac{k+4\epsilon k}{n}|U|$ otherwise.
Vertices $u_1, u_2 \in U$ have co-degree at least $|V|$ and
vertices $v_1, v_2 \in V$ have co-degree at least $|U|$.
All other pairs of vertices have co-degree at least $|U|$ or $|V|$.
The minimum co-degree is therefore at least
\[
\min\left\{\frac{k(1-4\epsilon)}{k(3/2+\epsilon)}|U|, |U|, |V| \right\}
     = \left\{\frac{2-8\epsilon}{3+2\epsilon}\cdot\frac{3}{5}N, \frac{3}{5}N, \frac{2}{5}N  \right\}.
\]
For some choice of $\epsilon$, this is approximately $\frac{2}{5}|V(H)|$.
\end{proof}

\newtheorem{cofanoquestion}[thmctr]{Question}
\begin{cofanoquestion}
What is the co-chromatic threshold of the Fano-free hypergraphs?
It is between $2/5$ and $1/2$.
\end{cofanoquestion}

\section{Acknowledgments}
We would like to thank the referee for detailed and insightful feedback, particularly for pointing out some flaws in our original proof of Theorem~\ref{coloringthm}.  Repairing this issue led us to a much improved proof approach.

\bibliographystyle{abbrv}
\bibliography{john-refs,extra-refs}

\end{document}